\documentclass[11pt,american,preprint]{elsarticle}
\makeatletter
\def\ps@pprintTitle{%
 \let\@oddhead\@empty
 \let\@evenhead\@empty
 \def\@oddfoot{\centerline{\thepage}}%
 \let\@evenfoot\@oddfoot}
\makeatother

\usepackage[margin=1in]{geometry}
\usepackage[T1]{fontenc}
\usepackage[latin9]{luainputenc}
\usepackage{amsmath}
\usepackage{amsthm}
\usepackage{amssymb}
\usepackage{stmaryrd}
\usepackage{esint}
\usepackage{nameref}
\usepackage{hyperref} 
\usepackage{cleveref} 
\usepackage[shortlabels]{enumitem}

\usepackage{ mathrsfs }

\usepackage{graphicx,amssymb}

\usepackage{array}
\newcolumntype{M}[1]{>{\centering\arraybackslash}m{#1}}

\usepackage{bbm} 

\theoremstyle{plain}
\newtheorem{thm}{\protect\theoremname}[section]

\theoremstyle{plain*}
\newtheorem*{thm*}{\protect\theoremname}

\theoremstyle{plain}
\newtheorem{lem}[thm]{\protect\lemmaname}
  
\theoremstyle{plain*}
\newtheorem*{lem*}{\protect\lemmaname}  
  
  \theoremstyle{plain}
  \newtheorem{prop}[thm]{\protect\propositionname}
  
    \theoremstyle{plain*}
  \newtheorem*{prop*}{\protect\propositionname}

    \theoremstyle{remark}
  \newtheorem{question}[thm]{Question}
   \theoremstyle{remark*}
 \newtheorem*{question*}{Question} 
  \theoremstyle{remark}
  \newtheorem{rem}[thm]{\protect\remarkname}
  \theoremstyle{remark*}
 \newtheorem*{rem*}{\protect\remarkname}
  \theoremstyle{plain}
  \newtheorem{cor}[thm]{\protect\corollaryname}
  
  \providecommand{\corollaryname}{Corollary}
  
\theoremstyle{definition}
  
\makeatother

\theoremstyle{plain} 
\newcommand{\thistheoremname}{}
\newtheorem{genericthm}[thm]{\thistheoremname}

\newtheorem*{genericthm*}{\thistheoremname}
\newenvironment{namedthm*}[1]
  {\renewcommand{\thistheoremname}{#1}%
   \begin{genericthm*}}
  {\end{genericthm*}}

\makeatother

\usepackage{babel}
 \providecommand{\lemmaname}{Lemma}
  \providecommand{\propositionname}{Proposition}
  \providecommand{\remarkname}{Remark}
\providecommand{\theoremname}{Theorem}

\newcommand{\plimgG}[3]{\mathop {#1\text{-lim}}_{#2\in #3}\,}

\newcommand{\R}{\mathbb{R}}
\newcommand{\N}{\mathbb{N}}
\newcommand{\Q}{\mathbb{Q}}
\newcommand{\Z}{\mathbb{Z}}

\newcommand{\rsupseteq}{\rotatebox[origin=c]{90}{$\subseteq$}}
\newcommand{\rsupsetneq}{\rotatebox[origin=c]{90}{$\subsetneq$}}

\newcommand{\requal}{\rotatebox[origin=c]{-90}{$=$}}

\usepackage{imakeidx}
\makeindex[columns=3, title=Index, intoc]

\usepackage[dvipsnames]{xcolor}
\usepackage{soul}

\usepackage{changepage}

\begin{document}

\begin{frontmatter}
\title{Iterated differences sets, diophantine approximations and applications}
\author[add1]{Vitaly Bergelson}
\ead{vitaly@math.ohio-state.edu}
\author[add1]{Rigoberto Zelada}
\ead{zeladacifuentes.1@osu.edu}

\address[add1]{Department of Mathematics. The Ohio State University, Columbus, OH 43210, USA}

\begin{abstract}
Let $v$ be an odd real polynomial (i.e. a polynomial of the form $\sum_{j=1}^\ell a_jx^{2j-1}$). We utilize sets of iterated differences to establish new 
results about sets of the form $\mathcal R(v,\epsilon)=\{n\in\N\,|\,\|v(n)\|{<\epsilon\}}$ where $\|\cdot\|$ denotes the distance to the 
closest integer. We then apply the new diophantine results to obtain applications to ergodic theory and combinatorics. In particular, 
we obtain a new characterization of weakly mixing systems as well as a new variant of Furstenberg-S{\'a}rk{\"o}zy theorem.  
\end{abstract}


\end{frontmatter}

\tableofcontents
\section{Introduction}
The goal of this paper is to establish new results pertaining to diophantine inequalities involving odd real polynomials and to obtain some applications to combinatorial number theory and ergodic theory.\\ 

Assume that $v$ is a real polynomial, with $\deg(v)\geq 1$, satisfying $v(0)=0$ and let $\epsilon>0$. Consider the set
\begin{equation}
    \mathcal R(v,\epsilon)\index{$\mathcal R(v,\epsilon)$}=\{n\in\N=\{1,2,...\}\,|\,\|v(n)\|<\epsilon\},
\end{equation}
where $\|\cdot\|$\index{$\|\cdot\|$} denotes the distance to the nearest integer.\\
It is well known that sets of the form $\mathcal R(v,\epsilon)$ are large in more than one sense. For example, it follows from Weyl's  
equidistribution theorem (see \cite{weyl1916Mod1}) that $\mathcal
R(v,\epsilon)$ has positive natural density. One can also show that $\mathcal R(v,\epsilon)$ is syndetic \index{syndetic} (\cite[Theorem 1.21]{FBook}), meaning 
that finitely many translations of $\mathcal R(v,\epsilon)$ cover $\N$ (i.e. $\mathcal R(v,\epsilon)$ has "bounded  gaps"). As a matter of fact, the sets $\mathcal R(v,\epsilon)$ posses a stronger property which is called \text{\rm{IP$^*$}}\index{IP$^*$}. A set $E\subseteq \N$ is called an IP set if it contains a set of the form 
$$\text{FS}((n_k)_{k\in\N})\index{FS$((n_k)_{k\in\N})$}=\{n_{k_1}+\cdots +n_{k_m}\,|\,k_1<\cdots<k_m;\text{ }m\in\N\},$$
for some increasing sequence $(n_k)_{k\in\N}$. A set $E\subseteq\N$ is IP$^*$ if it has a non-trivial intersection with every IP set.\footnote{
 IP$^*$  sets form a dual family in the sense of \cite[Chapter 9]{FBook}.
}\\ 
One can show with the help of Hindman's theorem\footnote{
Hindman's theorem states that if $E\subseteq \N$ is an IP set and $C_1,...,C_r\subseteq \N$ are such that $E = \bigcup_{j=1}^r C_j$, then there exists $s\in\{1,...,r\}$ such that $C_s$ is an IP set (see \cite{HIPPartitionRegular}). 
} 
that IP$^*$ sets have the finite intersection property, meaning that if  $E_1,...,E_m\subseteq \N$ are IP$^*$ sets, then $\bigcap_{j=1}^mE_j$  is also IP$^*$.\\

When $v$ is linear, $\mathcal R(v,\epsilon)$ has an even stronger property than  IP$^*$.   namely that of $\Delta^*$\index{$\Delta^*$}.
A set $E\subseteq \N$ is called a $\Delta^*$\index{$\Delta^*$} set if for any increasing sequence $(n_k)_{k\in\N}$, there exist $i<j$ for which
$$n_j-n_i\in E.$$
It is not hard to show that every $\Delta^*$ set is IP$^*$. Moreover, the family of IP$^*$ sets strictly contains the family of $\Delta^*$  sets. For example, the set 
$$\N\setminus\{2^j-2^i\,|\,i,j\in\N,\,i<j\}$$ 
is IP$^*$ but not $\Delta^*$ (See
\cite[pp. 165]{BergelsonErdosDifferences}).\\
One can show, with the help of Ramsey's Theorem,  that $\Delta^*$ sets have the finite intersection property (see \cite[pp.179]{FBook}). This implies, in particular,  that for any $\alpha_1,...,\alpha_m\in \R$ and any $\epsilon>0$, the set $\bigcap_{j=1}^m\{n\in\N\,|\,\|n\alpha_j\|<\epsilon\}$ is $\Delta^*$.\\

Unfortunately, for polynomials of degree two, the sets $\mathcal R(v,\epsilon)$ are no longer $\Delta^*$ (see, for example, \cite[pp.177-178]{FBook}). One is tempted to conjecture that the $\Delta_2^*$ sets, namely sets intersecting any set of the form
\begin{equation}\label{0.SecondDifferences}
\{(n_{k_4}-n_{k_3})-(n_{k_2}-n_{k_1})\,|\,k_4>k_3>k_2>k_1\},
\end{equation}
could be useful in dealing with polynomials of degree two and the corresponding sets $\mathcal R(v,\epsilon)$.
However, one can show, by using a natural modification of the construction in \cite{FBook}, that there exists $\epsilon>0$ such that for any irrational $\alpha$, the set  $\{n\in\N\,|\,\|n^2\alpha\|<\epsilon\}$ is not a $\Delta_2^*$ set.\\
To see this, fix an irrational number $\alpha$ and let $(n_k)_{k\in\N}$ be an increasing sequence in $\N$ such that 
\begin{equation}\label{0.SquareAndLimitOfSquares}
\lim_{k\rightarrow\infty}\|n_k\alpha\|=0
\text{ and }
\lim_{k\rightarrow\infty}\|n_k^2\alpha-\frac{1}{3}\|=0.\footnote{
The existence of such a sequence $(n_k)_{k\in\N}$ follows from  \cite[Theorem 1.011]{HardyLittlewood1914some}. One can also use, for example,  the two-dimensional version of Weyl's equidistribution theorem \cite{weyl1916Mod1}. Finally, one could also invoke the fact that the transformation $T:\mathbb T^2\rightarrow \mathbb T^2$ defined by $T(x,y)=(x+\alpha,y+2x+\alpha)$ is minimal. See for example \cite[Lemma 1.25]{FBook}.}
\end{equation}
By passing, if needed, to a subsequence, we can also assume that for any $j,k\in\N$ with $j<k$,
\begin{equation}\label{0.TwoTupleLimit}
\|n_jn_k\alpha\|<\frac{1}{k}.
\end{equation}
So, for any large enough and distinct $j,k\in\N$, we have  $\|n_jn_k\alpha\|<\dfrac{\epsilon}{16}$ and $\|n_k^2\alpha-\frac{1}{3}\|<\dfrac{\epsilon}{16}$. It follows by a simple calculation that for large enough $k_4>k_3>k_2>k_1$, 
$$\|[(n_{k_4}-n_{k_3})-(n_{k_2}-n_{k_1})]^2\alpha-\frac{4}{3}\|<\epsilon,$$
which implies that the set $\mathcal R(n^2\alpha,\frac{1}{6})$ is not $\Delta_2^*$.\\

It comes as a pleasant surprise that  $\Delta_2^*$ sets work well with the sets $\mathcal R(n^3\alpha,\epsilon)$. 
\begin{prop}\label{0.CubicCase}
For any real number  $\alpha$ and any $\epsilon>0$, the set 
$$\mathcal R(n^3\alpha,\epsilon)=\{n\in\N\,|\,\|n^3\alpha\|<\epsilon\}$$
is $\Delta_2^*$.
\end{prop}
It turns out that \cref{0.CubicCase} generalizes nicely to odd real polynomials, namely polynomials of the form
\begin{equation}\label{0.OddPolynomial}
v(x)=\sum_{j=1}^\ell a_j x^{2j-1}.
\end{equation}
(Note that a real polynomial $v$ satisfies $v(-x)=-v(x)$ if and only if $v$ is of the form \eqref{0.OddPolynomial}).\\
Before formulating a generalization of \cref{0.CubicCase} to odd polynomials of arbitrary degree, we have to introduce the family of $\Delta_\ell^*$ sets, $\ell\in\N$.\\
Define  the function $\partial:\bigcup_{\ell\in\N}\Z^{2^\ell}\rightarrow\Z$\index{$\partial(m_1,...,m_{2^\ell})$} recursively by the formulas:
\begin{enumerate}
    \item $\partial(m_1,m_2)=m_2-m_1$.
    \item $\partial(m_1,...,m_{2^\ell})=\partial(m_{2^{\ell-1}+1},...,m_{2^\ell})-\partial(m_1,...,m_{2^{\ell-1}}),$ $\ell>1$.
\end{enumerate}
Given $\ell\in\N$, we will say that a set $E\subseteq \N$ is $\Delta_\ell^*$\index{$\Delta_\ell^*$} if for any 
increasing sequence $(n_k)_{k\in\N}$ in $\N$, there exist $$k_1<k_2<k_3<\cdots<k_{2^\ell}$$ for which
$$\partial (n_{k_1},...,n_{k_{2^\ell}})\in E.\footnote{
 Note that the notion of $\Delta_1^*$ is the same as the notion of $\Delta^*$\index{$\Delta^*$} defined above.
}$$
For example, a set $E\subseteq \N$ is  $\Delta_3^*$ if for any  increasing sequence $(n_k)_{k\in\N}$ in $\N$, there exist $k_1<\cdots<k_8$ for which
$$[(n_{k_8}-n_{k_7})-(n_{k_6}-n_{k_5})]-[(n_{k_4}-n_{k_3})-(n_{k_2}-n_{k_1})]\in E.$$
 One can show  that for each $\ell\in\N$, $\Delta_\ell^*$ sets have the finite intersection property. (See Section \ref{SectionBetaN} for more information on $\Delta_\ell^*$ sets.)\\
 
We are now in position to state a generalization of \cref{0.CubicCase}.
\begin{thm}\label{0.OddDegreeRecurrence}
For any odd real polynomial $v(x)=\sum_{j=1}^{\ell}a_jx^{2j-1}$ and any $\epsilon>0$, the set $$\mathcal R(v,\epsilon)=\{n\in\N\,|\,\|v(n)\|<\epsilon\}$$ 
is $\Delta_\ell^*$.
\end{thm}
\begin{rem}
One can show that for $\ell>1$, the families IP$^*$ and $\Delta_\ell^*$ are, so to say, in general position. Namely, IP$^*\not\subseteq\Delta_\ell^*$ (see \cref{3.Delta^lRichNoIPs}) and $\Delta_\ell^*\not\subseteq\text{IP}^*$ (see  \cref{3.IPWithNoDelta2}).
\end{rem}
The following theorem shows that odd real polynomials are, roughly, the only polynomials for which the sets $\mathcal R(v,\epsilon)$ are always  $\Delta_\ell^*$:
\begin{thm}\label{0.CharacterizationOfOddPolynoials}
Let $\ell\in\N$ and let $v(x)$ be a real polynomial. The set $\mathcal R(v,\epsilon)$ is $\Delta_\ell^*$ for any $\epsilon>0$ if and only if there exists a polynomial $w\in\Q[x]$ with $w(0)\in\Z$ and such that $v-w$ is an odd polynomial of degree at most $2\ell-1$. 
\end{thm}
There are two basic approaches to the proof of \cref{0.OddDegreeRecurrence}. The first approach is based on the inductive utilization of (the finite) Ramsey Theorem. The second approach uses a special family of ultrafilters in $\beta\N$ which is of interest in its own right and has not been utilized before in a similar context. Each of these approaches has its own pros and cons.\\
The first approach allows to formulate and prove a finitistic version of \cref{0.OddDegreeRecurrence} (this is a pro), but the proof gets quite cumbersome (this is a con). This approach is carried out in Subsection \ref{SubsectionFinitisticProof}.\\
The second approach, which is implemented in Subsection \ref{SubsectionUltrafilterApproach}, has the advantage of being shorter and much easier to follow. The disadvantage of this approach seems to be mostly lying with the fact that some readers may not be familiar with ultrafilters.  We remedy this by giving detailed definitions and some of the necessary background in Section \ref{SectionBetaN}.\\
It is worth mentioning that we will also utilize  the ultrafilter technique in the proofs of \cref{0.CharacterizationOfOddPolynoials} (see Section \ref{SectionCharPol}) and of a converse to  \cref{0.OddDegreeRecurrence} (see Section \ref{SectionAConverseResult}).\\

In Section \ref{SecHilbert}, we deal with applications to unitary actions. In particular, we establish the following result.
\begin{thm}\label{0.CompactHilbert}
Let $U:\mathcal H\rightarrow \mathcal H$ be a unitary operator and let $v(x)=\sum_{j=1}^\ell a_jx^{2j-1}$ be a non-zero odd polynomial with $v(\Z)\subseteq\Z$. The following are equivalent:
\begin{enumerate}[(i)]
\item $U$ has discrete spectrum (i.e. $\mathcal H$ is spanned by eigenvectors of $U$).
\item For any $f\in\mathcal H$ and any $\epsilon>0$, the set 
$$\{n\in\N\,|\,\|U^{v(n)}f-f\|_{\mathcal H}<\epsilon\}$$
is $\Delta_\ell^*$.
\end{enumerate}
\end{thm}
 \cref{0.CompactHilbert} has the following ergodic-theoretical corollary.
\begin{cor}\label{0.CompactMeasurable}
Let $(X,\mathcal A,\mu)$ be a probability space\footnote {Throughout this paper we will assume that the probability spaces we deal with are standard,
that is, isomorphic mod 0 to a disjoint union of an interval equipped
with the Lebesgue measure and a countable number of atoms.}
and let $T:X\rightarrow X$ be an ergodic invertible probability measure preserving transformation. The following are equivalent:
\begin{enumerate}[(i)]
\item $(X,\mathcal A,\mu, T)$ is isomorphic to a translation on a compact abelian group.
\item For any odd polynomial $v(x)=\sum_{j=1}^\ell a_jx^{2j-1}$ with $v(\Z)\subseteq\Z$, any  $A\in\mathcal A$ and any $\epsilon>0$, the set 
$$\{n\in\N\,|\, \mu(A\cap T^{-v(n)}A)>\mu(A)-\epsilon\}$$
is $\Delta_\ell^*$.
\item There exists a non-zero odd polynomial $v(x)=\sum_{j=1}^\ell a_jx^{2j-1}$ with $v(\Z)\subseteq\Z$ such that for any  $A\in\mathcal A$ and any $\epsilon>0$, the set 
$$\{n\in\N\,|\, \mu(A\cap T^{-v(n)}A)>\mu(A)-\epsilon\}$$
is $\Delta_\ell^*$.
\end{enumerate}
\end{cor}
Another application of \cref{0.CompactHilbert} to measure preserving systems requires the introduction of the notion of an \textit{almost} $\Delta_\ell^*$ set, denoted by A-$\Delta_\ell^*$\index{A-$\Delta_\ell^*$}. Recall that the upper Banach density of a set $E\subseteq \N$, $d^*(E)$, is defined by $$d^*(E)=\limsup_{N-M\rightarrow\infty}\frac{|E\cap\{M+1,...,N\}|}{N-M},$$
where, for a finite $F\subseteq \N$, $|F|$ denotes the cardinality of $F$.
 Given $\ell\in\N$, a set $D\subseteq \N$ is A-$\Delta_\ell^*$ if there exists a set $E\subseteq \N$ with $d^*(E)=0$\index{$d^*(E)$},
such that $D\cup E$ is $\Delta_\ell^*$. 
\begin{thm}\label{0.AlmostMeasurableCase}
Let $(X,\mathcal A,\mu,T)$ be an invertible probability measure preserving system and let  $v(x)=\sum_{j=1}^\ell a_jx^{2j-1}$ be an odd polynomial with $v(\Z)\subseteq\Z$. For any $A\in\mathcal A$ and any $\epsilon>0$, the set
\begin{equation}\label{0.SimpleSet}
\mathcal R_A(v,\epsilon)=\{n\in\N\,|\,\mu(A\cap T^{-v(n)}A)>\mu^2(A)-\epsilon\}
\end{equation}
is A-$\Delta_\ell^*$.
\end{thm}
\begin{rem}
It was shown in \cite{BFM} that the "sets of large returns" $\mathcal R_A(v,\epsilon)$ have the IP$^*$ property for any polynomial $v$ with $v(\Z)\subseteq\Z$ and satisfying $v(0)=0$.  It will be shown in Section \ref{SectionNotionsOfLargness} that for each $\ell\in\N$, there exists an IP$^*$ set which is not A-$\Delta_\ell^*$. So, \cref{0.AlmostMeasurableCase} provides  new information about sets of large returns when $v$ is an odd polynomial.
\end{rem}
We remark that the quantity $\mu^2(A)$ in \eqref{0.SimpleSet} is optimal (consider any strongly mixing system\footnote
{A probability measure preserving system $(X,\mathcal A,\mu,T)$ is strongly mixing if for any $A,B\in\mathcal A$,
$$\lim_{n\rightarrow\infty}\mu(A\cap T^{-n}B)=\mu(A)\mu(B).$$
}).\\
The following corollary of \cref{0.AlmostMeasurableCase} is a result in additive combinatorics which might be seen as a variant of the Furstenberg-S{\'a}rk{\"o}zy theorem (see \cite{sarkozy1978difference} and \cite[Theorem 3.16]{FBook}).
\begin{cor}\label{0.SarkozyLike}
Let $E\subseteq\N$ be such that $d^*(E)>0$
and let $v(x)=\sum_{j=1}^\ell a_jx^{2j-1}$ be an odd polynomial with $v(\Z)\subseteq\Z$. Then the set 
$$\{n\in\N\,|\,v(n)\in E-E\}$$
is A-$\Delta_\ell^*$.
\end{cor}
We also have a new recurrence property for weakly mixing systems. A probability measure preserving system $(X,\mathcal A,\mu,T)$ is weakly mixing if for any $A,B\in\mathcal A$,
$$\lim_{N\rightarrow\infty}\frac{1}{N}\sum_{j=1}^N|\mu(A\cap T^{-n}B)-\mu(A)\mu(B)|=0.$$
\begin{cor}\label{0.WeaklyMixingCase}
Let $v(x)=\sum_{j=1}^\ell a_jx^{2j-1}$ be a non-zero odd polynomial with $v(\Z)\subseteq\Z$. An invertible probability measure preserving system $(X,\mathcal A,\mu,T)$ is  weakly mixing if and only if  for any $A,B\in\mathcal A$ and any $\epsilon>0$, the set
$$\mathcal R_{A,B}(v,\epsilon)=\{n\in\N\,|\,|\mu(A\cap T^{-v(n)}B)-\mu(A)\mu(B)|<\epsilon\}$$
is A-$\Delta_\ell^*$.
\end{cor}
In Section \ref{SecExample}, we provide an example of a weakly mixing system $(X,\mathcal A,\mu,T)$ which shows that in the statement of \cref{0.WeaklyMixingCase},  A-$\Delta_\ell^*$ can not be replaced  by $\Delta_\ell^*$.\\

We conclude the introduction with formulating a recent result \cite{BerZelMixingOfAllorders} which demonstrates yet another connection between $\Delta_\ell^*$ sets and ergodic theory.  
\begin{thm}[Cf. \cite{KuangYeDeltaMixing}]
Let $(X,\mathcal A,\mu,T)$ be an invertible probability measure preserving system. The following are equivalent:
\begin{enumerate}[(i)]
    \item $(X,\mathcal A,\mu, T)$ is strongly mixing.
    \item There exists an $\ell\in\N$ such that for any  $A\in\mathcal A$ and any $\epsilon>0$, the set
    $$\{n\in\N\,|\, |\mu(A\cap T^{-n}A)-\mu^2(A)|<\epsilon\}$$
    is $\Delta_\ell^*$.
    \item For any $\ell\in\N$, any $A_0,...,A_{\ell+1}\in\mathcal A$ and any $\epsilon>0$, the set 
     $$\{n\in\N\,|\, |\mu(A_0\cap T^{-n}A_1 \cdots\cap T^{-(\ell+1)n}A_{\ell+1})-\prod_{j=0}^{\ell+1}\mu(A_j)|<\epsilon\}$$
    is $\Delta_\ell^*$.
\end{enumerate}
\end{thm}
The structure of the paper is as follows. In \Cref{SectionBetaN}, we provide the necessary background on ultrafilters and establish the connection between ultrafilters and $\Delta_\ell^*$ sets. In Section 3, we prove \cref{0.OddDegreeRecurrence} as well as its finitistic version. In  \Cref{SectionAConverseResult}, we  prove a converse to \cref{0.OddDegreeRecurrence}. In  \Cref{SectionCharPol}, we prove \cref{0.CharacterizationOfOddPolynoials}. In \Cref{SecHilbert}, we focus on  applications to unitary actions. In \Cref{SecExample}, we provide an example of a weakly mixing system which demonstrates that  \cref{0.WeaklyMixingCase} can not be improved. In \Cref{SectionNotionsOfLargness}, we discuss  the relations between the various families of subsets of $\N$ that we deal with throughout this paper.
\section{$\beta\N$ and $\Delta_\ell^*$ sets}\label{SectionBetaN}
In this section we provide some background on the space of ultrafilters $\beta\N$ and connect the notion of  $\Delta_\ell^*$ with a natural family in $\beta\N$.\\
Let $p$ be a family of subsets of $\N$. We say that $p$ is a \textbf{filter} if it has the following properties:
\begin{enumerate}[(i)]
\item $\emptyset\not\in p$ and $\N\in p$.
\item If $A,B\in p$, then $A\cap B\in p$.
\item If $A\in p$ and $A\subseteq B$, then $B\in p$.
\end{enumerate}
If, in addition, $p$ satisfies
\begin{enumerate}[(iv)]
    \item  for any  $A,B\subseteq \N$, if $A\cup B\in p$ and $A\not\in p$,  then $B\in p$.
\end{enumerate}
then we say that $p$ is an \textbf{ultrafilter}.  
 It is not hard to show that an ultrafilter $p$ is a \textit{maximal filter}, meaning that $p$ is not properly contained in any other filter. The set of all ultrafilters on $\N$ is denoted by $\beta\N$.\\
One can introduce a natural topology on $\beta\N$: given $A\subseteq \N$, let $$\overline A=\{p\in\beta\N\,|\,A\in p\}.$$ The family $\{\overline A\,|\,A\subseteq\N\}$ forms a basis for the open sets (and a basis for the closed sets) for this topology. With this topology, $\beta\N$ becomes a compact Hausdorff space. Identifying $n\in\N$ with the \textbf{principal ultrafilter}\index{principal ultrafilter} $\overline n=\{A\subseteq\N\,|\,n\in A\}$ allows us to interpret $\beta\N$ as a representation of the Stone-\v{C}ech compactification of $\N$. We remark in passing that the cardinality of $\beta\N$ is that of $\mathcal P(\mathcal P(\N))$ (and so, $\beta\N$ is a non-metrizable topological space).\\
An alternative way of looking at $\beta\N$ is to identify each ultrafilter $p\in\beta\N$ with a finitely additive, $\{0,1\}$-valued probability measure $\mu_p$ on the power set $\mathcal P (\N)$. The measure $\mu_p$ is naturally defined by the condition $\mu_p(A)=1$ if and only if $A\in p$. In this way, we can say that $A\subseteq \N$ is $p$-large whenever $\mu_p(A)=1$ (or equivalently, if $A\in p$).\\  
One can naturally extend the operation + from $\N$ to an associative binary operation $+:\beta\N\times\beta\N\rightarrow\beta\N$ by defining $p+q$ to be the unique ultrafilter such that $A\in p+q$ if and only if
\begin{equation}\label{1.SumFormula}
\{n\in\N\,|\,-n+A\in q\}\in p
\end{equation}
(the set $-n+A$ is defined by $m\in(-n+A)$ if and only if $n+m\in A$).\\
With the operation $+$, $\beta\N$ becomes a compact right topological semigroup (meaning that the function $\rho_p:\beta\N\rightarrow\beta\N$, defined by $\rho_p(q)=q+p$ is continuous). \\ 
In a similar way, one can define $(\beta\Z,+)$ (This kind of construction actually  works for any discrete semigroup. For more on the Stone-\v{C}ech compactification of a discrete semigroup see \cite{HBook}). Note that $(\beta\N,+)$ is a closed sub-semigroup of $(\beta\Z,+)$.\\

For each non-principal ulltrafilter $p\in\beta\N$, the family of subsets of $\N$
\begin{equation}\label{1.DeltaPDef}
\{A\subseteq\N\,|\,\{n\in\N\,|\,n+A\in p\}\in p\}
\end{equation}
is again a non-principal ultrafilter, which we denote by $-p+p$. Note that the notation $-p+p$ for the ultrafilter defined by \eqref{1.DeltaPDef} has to be taken with a grain of caution. To justify the notation $-p+p$ observe that given a non-principal 
ultrafilter $p\in\beta\N$, one can naturally  define the ultrafilter $-p\in\Z^*=\beta\Z\setminus\Z$ by the rule 
$-A\in p$ if and only if $A\in-p$. Now, it is not hard to check that 
$\N^*=\beta\N\setminus\N$ is a left ideal of the semigroup $(\beta\Z,+)$, and so, if $p\in\N^*$, then $-p+p\in\N^*$.\\

Let $X$ be a topological space and let $p\in\beta\N$ be an ultrafilter. Given a sequence $(x_k)_{k\in\N}$, we will write
$$\plimgG{p}{n}{\N}x_n=x\index{$\plimgG{p}{n}{\N}x_n$}$$
if for any neighborhood $U$ of $x$
$$\{n\in\N\,|\,x_n\in U\}\in p.$$
It is easy to see that $\plimgG{p}{n}{\N} x_n$ exists and is unique in any compact Hausdorff space.\\ 
The proof of the following useful lemma is similar to the proof of Theorem 3.8 in \cite{ERTaU}. 
\begin{lem}\label{1.LemaDeltaEquality}
 Let $X$ be a compact Hausdorff space, let $p\in \beta \N$ be a non-principal ultrafilter and let $(x_k)_{k\in \N}$ be a sequence in $X$. Then 
\begin{equation}\label{1.EquationDeltaEquality}
\plimgG{(-p+p)}{n}{\N} x_n=\plimgG{p}{m}{\N} \plimgG{p}{n}{\N} x_{n-m}.
\end{equation}
\end{lem}
\begin{proof}
For a non-empty open set $U\subseteq X$, let
$$A_U=\{n\in \N\,|\,x_n\in U\}.$$
For any $m\in \N$, let 
$$B_U(m)=\{n\in \N\,|\,n>m\text{ and }x_{n-m}\in U\}.$$
Note that for each $m\in \N$, 
$$B_U(m)=(m+A_U).$$
So, by \eqref{1.DeltaPDef}, $A_U\in  -p+p$ if and only if  $\{m\in \N\,|\, B_U(m)\in p\}\in p$. Hence,
$$\plimgG{(-p+p)}{n}{\N} x_n=\plimgG{p}{m}{\N} \plimgG{p}{n}{\N} x_{n-m}.$$
\end{proof}
In what follows we will need an extension of \cref{1.LemaDeltaEquality} for "iterated differences" of ultrafilters which are defined for any $\ell\in\N$ and any non-principal ultrafilter $p\in\beta\N$ by  
$$p_\ell=-p_{\ell-1}+p_{\ell-1},\index{$p_\ell$}$$
where, by convention, $p_0=p$. (Note that for any $\ell\in\N$ and any $t\leq \ell$, $p_\ell=(p_{\ell-t})_{t}$
).
Before formulating this extension, let us recall the recursive definition of the map  $\partial:\bigcup_{\ell\in\N} \Z^{2^\ell}\rightarrow \Z$ which was introduced in the Introduction:
\begin{enumerate}
    \item For $(n_1,n_2)\in\Z^2$, $\partial(n_1,n_2)=n_2-n_1$.
    \item For each $\ell>1$ and any $(n_1,...,n_{2^\ell})\in\Z^{2^\ell}$, $$\partial(n_1,...,n_{2^\ell})=\partial(n_{2^{\ell-1}+1},...,n_{2^\ell})-\partial(n_1,...,n_{2^{\ell-1}}).$$
\end{enumerate}
For instance $\partial(1,2)=2-1=1$, $\partial(5,3)=3-5=-2$ and $\partial(1,2,5,3)=\partial(5,3)-\partial(1,2)=-3$.\\
By induction on $\ell\geq 2$, one can show that for any $n_1,...,n_{2^\ell}\in\Z$,
\begin{equation}\label{1.TechnicalIdentity}
\partial(n_1,...,n_{2^\ell})=\partial(\partial(n_1,n_2),...,\partial(n_{2^\ell-1},n_{2^\ell})).
\end{equation}
To verify \eqref{1.TechnicalIdentity}, one just needs to note that for any $\ell\geq 3$,
\begin{multline*}
\partial(\partial(n_1,n_2),...,\partial(n_{2^\ell-1},n_{2^\ell}))\\
=\partial(\partial(n_{2^{\ell-1}+1},n_{2^{\ell-1}+2}),...,\partial(n_{2^\ell-1},n_{2^\ell}))-\partial(\partial(n_1,n_2),...,\partial(n_{2^{\ell-1}-1},n_{2^{\ell-1}})).
\end{multline*}
We are in position now to formulate the desired extension of \cref{1.LemaDeltaEquality}, the proof of which can be done by routine  induction with the help of \eqref{1.TechnicalIdentity}.
\begin{lem}\label{1.LemaIteratedDifLimit}
Let $X$ be a compact Hausdorff space, let  $p\in \beta \N$ be a non-principal ultrafilter and let $(x_k)_{k\in \N}$ be a sequence in $X$. Then for each $\ell\in\N$,
$$\plimgG{p_\ell}{m}{\N}x_m=\plimgG{p}{m_1}{\N}\cdots\plimgG{p}{m_{2^\ell}}{\N}x_{\partial(m_1,...,m_{2^\ell})}.$$
\end{lem}
\begin{proof}
\begin{multline*}
\plimgG{p_\ell}{m}{\N}x_m=\plimgG{p_{\ell-1}}{m_1}{\N}\plimgG{p_{\ell-1}}{m_2}{\N}x_{\partial(m_1,m_2)}\\
=\plimgG{p_{\ell-2}}{m_1}{\N}\plimgG{p_{\ell-2}}{m_2}{\N}\plimgG{p_{\ell-2}}{m_3}{\N}\plimgG{p_{\ell-2}}{m_4}{\N}x_{\partial(\partial(m_1,m_2),\partial(m_3,m_4))}\\
=\plimgG{p_{\ell-2}}{m_1}{\N}\plimgG{p_{\ell-2}}{m_2}{\N}\plimgG{p_{\ell-2}}{m_3}{\N}\plimgG{p_{\ell-2}}{m_4}{\N}x_{\partial(m_1,m_2,m_3,m_4)}\\
=\cdots=\plimgG{p}{m_1}{\N}\cdots\plimgG{p}{m_{2^\ell}}{\N}x_{\partial(m_1,...,m_{2^\ell})}.
\end{multline*}
\end{proof}

Now we turn our attention to $\Delta_\ell$ sets, $\ell\in\N$. When $\ell=1$, $E\subseteq \N$ is a $\Delta_1$ set (or $\Delta$\index{$\Delta$} set for simplicity) if  there exists an increasing sequence $(n_k)_{k\in\N}$ in $\N$ with the property that
$$\{n_j-n_i\,|\,i<j\}\subseteq E.$$
The following result, which establishes the connection between ultrafilters of the form $-p+p$ and $\Delta$ sets, is a version of Lemma 3.12 in \cite{BerHindQuotientSets}.
\begin{prop}\label{1.DeltaIfandonlyIfDifferenceUltra}
Let $A\subseteq\N$. There exists a non-principal ultrafilter $p\in\beta\N$ such that $A\in-p+p$ if and only if there exists an increasing sequence $(n_k)_{k\in\N}$ in $\N$ such that 
$$\{n_j-n_i\,|\,i<j\}\subseteq A.$$
\end{prop}
Given $\ell\in\N$, a set $E\subseteq \N$ is a \textbf{$\Delta_\ell$\index{$\Delta_\ell$} set} if it contains the $\ell$-th differences set generated by an increasing sequence $(n_k)_{k\in\N}$ in $\N$. Given a sequence $(n_k)_{k\in\N}$ in $\Z$, the \textbf{$\ell$-th differences set generated by $(n_k)_{k\in\N}$} is the set defined by
\begin{equation}\label{1.DeltaSetDef}
D_\ell((n_k)_{k\in\N})=\{\partial(n_{j_1},...,n_{j_{2^\ell}})\,|\,j_1<\cdots<j_{2^\ell}\}.\index{$D_\ell((n_k)_{k\in\N})$}
\end{equation}
The following theorem forms a natural extension of \cref{1.DeltaIfandonlyIfDifferenceUltra} to $\Delta_\ell$ sets.
\begin{thm}\label{1.DeltaCharacterization}
Let $d\in\N$ and let $A_0,...,A_d\subseteq \N$. The following are equivalent:
\begin{enumerate}[(i)]
\item There exists a non-principal ultrafilter $p\in \beta \N$ such that for each $\ell\in \{0,...,d\}$, $A_{\ell}\in p_\ell$.
\item There is an increasing sequence $(n_k)_{k\in\N}$ in $\N$ such that for each $\ell\in\{0,...,d\}$, 
$$D_\ell((n_k)_{k\in\N})\subseteq A_\ell,$$
where by convention $D_0((n_k)_{k\in\N})=\{n_k\,|\,k\in\N\}$.
\item There exist infinite sets $I_0,...,I_d\subseteq \N$ such that for each $\ell\in\{0,...,d\}$, $I_{\ell}\subseteq A_\ell$ and if $\ell<d$, we have that for all $n\in I_\ell$, the cardinality of $$I_{\ell}\setminus(n+ I_{\ell+1})$$
is finite\footnote{Cf. Definition 1.5 in \cite{BerHindAbundant}.
}.
\end{enumerate}
\end{thm}
\begin{proof}
(i)$\implies$(ii): We will consider the compact Hausdorff space $X=\{0,1\}^{\N\cup\{0\}}$ with the continuous map $\sigma:X\rightarrow X$ given by 
$$(\sigma(x))(n)=x(n+1).$$
Note that for each $\ell\in\{0,...,d\}$ there exists $x_\ell\in X$ such that 
$$\plimgG{ p_\ell}{n}{\N}\sigma^{n} \mathbbm 1_{A_\ell} =x_\ell.$$
So, since for any $n\in A_\ell$,
$$\sigma^n\mathbbm 1_{A_\ell}(0)=\mathbbm 1_{A_\ell}(n)=1$$
and  $A_\ell\in p_\ell$, we have that $x_\ell(0)=1$.\\
Now observe that by \cref{1.LemaIteratedDifLimit} we have 
\begin{equation*}\label{1.Delta()Limit1}
\plimgG{p}{n_1}{\N}\cdots\plimgG{p}{n_{2^\ell}}{\N} \sigma^{\partial(n_1,...,n_{2^\ell})}\mathbbm 1_{A_\ell}=x_\ell,
\end{equation*}
which implies that for any $\ell\in\{0,...,d\}$ there exists a $B_1^\ell\in p$ such that for  any $r\in\{2,...,2^\ell\}$, there exists a set $B_r^\ell(n_1,...,n_{r-1})\in p$ defined recursively for $n_1\in B_1^\ell$, $n_2\in B_2^\ell(n_1)$,..., $n_{r-1}\in B_{r-1}^\ell(n_1,...,n_{r-2})$ with the property that if $n_1\in B_1^\ell$ and for each $r\in\{2,...,2^\ell\}$, $n_r\in B_r^\ell(n_1,...,n_{r-1})$, then
$$\sigma^{\partial(n_1,...,n_{2^\ell})}\mathbbm 1_{A_\ell}(0)=x_{\ell}(0)=1.$$
So, $\partial(n_1,...,n_{2^\ell})\in A_\ell$.\\

With the definition of $B_r^\ell(n_1,...,n_{r-1})$ in mind and adhering to the convention that $\bigcap _{j\in\emptyset}F_j=\N$, we can pick the sequence $(n_k)_{k\in\N}$ inductively as follows:\\
First, pick
\begin{equation}\label{1.ProofLemmaFirstInduction}
n_1\in\bigcap_{\ell=0}^d B_1^\ell\in p.
\end{equation}
For $t\leq 2^d-1$, pick 
\begin{equation}\label{1.ProofLemmaFiniteInduction}
n_{t+1}\in\bigcap_{\ell=0}^d\left[B_1^\ell\cap \bigcap_{s\in\{1,...,\min\{2^\ell,t\}\}\setminus\{2^\ell\}}\bigcap_{1\leq j_1<\cdots<j_s\leq t}B_{s+1}^\ell(n_{j_1},...,n_{j_s})\right]\in p.
\end{equation}
Finally, for $t\geq 2^d$, pick
\begin{equation}\label{1.ProofLemmaUnboundedInduction}
n_{t+1}\in\bigcap_{\ell=0}^d\left[B_1^\ell\cap \bigcap_{s\in\{1,...,2^\ell\}\setminus\{2^\ell\}}\bigcap_{1\leq j_1<\cdots<j_s\leq t}B_{s+1}^\ell(n_{j_1},...,n_{j_s})\right]\in p.
\end{equation}
Since the sets in \eqref{1.ProofLemmaFirstInduction}, \eqref{1.ProofLemmaFiniteInduction} and \eqref{1.ProofLemmaUnboundedInduction} are members of $p$ and $p$ is a non-principal ultrafilter, we can assume without loss of generality, that for each $k\in\N$, $n_{k+1}>n_{k}$; which completes the proof of (i)$\implies$(ii).\\ 

(ii)$\implies$(iii): Let $(n_k)_{k\in\N}$ be an increasing  sequence of natural numbers
such that for each $\ell\in\{0,...,d\}$,
$$D_\ell((n_k)_{k\in\N})\subseteq A_\ell.$$
For each $\ell\in \{0,...,d\}$, let 
\begin{equation}\label{1.defI_l}
    I_{\ell}=D_\ell((n_k)_{k\in\N}).
\end{equation}
Let $\ell\in\{0,...,d-1\}$. It follows from \eqref{1.defI_l} that 
$$I_{\ell+1}=\{\partial(n_{j_{2^\ell+1}},...,n_{j_{2^{\ell+1}}})-\partial(n_{j_1},...,n_{j_{2^\ell}})\,|\,j_1<\cdots<j_{2^{\ell+1}}\},$$
so, for any $n\in I_{\ell}$, $$I_{\ell}\setminus(n+ I_{\ell+1})$$
is a finite set. In particular, if $I_\ell$ is infinite, $I_{\ell+1}$ is infinite as well. So, since $(n_k)_{k\in\N}$ is increasing, we have that for each $\ell\in\{0,...,d\}$, $I_\ell$ is infinite. Hence (ii)$\implies$(iii) follows from the observation that, by \eqref{1.defI_l}, for each $\ell\in\{0,...,d\}$,
$$I_\ell=D_\ell((n_k)_{k\in\N})\subseteq A_\ell.$$
(iii)$\implies$(i): First note that since $I_{\ell}\subseteq A_{\ell}$ for any $\ell\in\{0,...,d\}$, it
suffices to show that there exists a non-principal $p\in \beta \N$ with the property that for each $\ell\in\{0,...,d\}$, $I_\ell\in p_\ell$.\\
There exists a non-principal ultrafilter $p\in\beta \N$ such that
$$I_0\in p_0.$$
We claim that for any $\ell\in \{1,...,d\}$, $I_\ell\in p_\ell$. To see this,  assume that 
$I_{\ell-1}\in p_{\ell-1}$.
Then, by (iii), we have that 
$$\{n\in \N\,|\,n+ I_{\ell}\in p_{\ell-1}\}\in p_{\ell-1}.$$
So, by the definition of $p_{\ell}$, $I_{\ell}\in p_{\ell}$. We are done.
\end{proof}
Recall that, given $\ell\in\N$,  a set $E\subseteq \N$ 
is a $\Delta_\ell^*$\index{$\Delta_\ell^*$} set if for any increasing sequence $(n_k)_{k\in\N}$, there exists $j_1<\cdots<j_{2^\ell}$ for which $\partial(n_{j_1},...,n_{j_{2^\ell}})\in E$. We have the following characterization of $\Delta_\ell^*$ sets.
\begin{prop}\label{1.WorkingDfn}
Given $\ell\in\N$, a set $E\subseteq\N$ is a $\Delta_\ell^*$ set if and only if $E$ has a non-trivial intersection with any $\Delta_\ell$ set.
\end{prop}
\begin{proof}
If $E$ is a $\Delta_\ell^*$ set, it is clear that it has a  non-trivial intersection 
with every $\Delta_\ell$ set.\\
Now suppose that $E$ has a non-trivial intersection with any $\Delta_\ell$ set. Let $(n_k)_{k\in\N}$ be an increasing sequence in $\N$. By passing to a subsequence, if needed, we can assume that for all $k\in\N$, $n_{k+1}\geq 2n_k$. It follows that 
$$D_\ell((n_{k})_{k\in\N})\subseteq\N.$$
But $D_\ell((n_{k})_{k\in\N})$ is itself a $\Delta_\ell$ set. Thus, since $E\cap D_\ell((n_{k})_{k\in\N})\neq \emptyset$, there exist $j_1<\cdots<j_{2^\ell}$ for which 
$\partial(n_{j_1},...,n_{j_{2^\ell}})\in E$, completing the proof.
\end{proof}
We record for the future use two immediate corollaries of \cref{1.DeltaCharacterization} and  \cref{1.WorkingDfn}:
\begin{cor}\label{1.DeltaL*characterization} Let $E\subseteq \N$ and let $\ell\in\N$. $E$ is a  $\Delta_\ell^*$ set if and only if $E\in p_{\ell}$ for any non-principal ultrafilter $p\in\beta\N$.
\end{cor}
\begin{cor}
For any $N\in\N$ and any $\Delta_\ell^*$ sets $E_1,...,E_N$, the set $E_1\cap\cdots\cap E_N$ is also  $\Delta_\ell^*$.
\end{cor}
We conclude this section by noting that if $D$ is a  $\Delta_\ell^*$  set, then it is syndetic (i.e. there exist $n_1,...,n_N\in\N$ such that $\N\subseteq \bigcup_{j=1}^ND-n_j$).
\begin{lem}\label{1.Delta*IsSyndetic}
For any $\ell\in\N$, any $\Delta_\ell^*$ set is syndetic.
\end{lem}
\begin{proof}
Let $\ell\in\N$. We will show that if $D\subseteq \N$ is not syndetic, then it is not $\Delta_\ell^*$. If $D$ is not syndetic, then there exist increasing sequences of natural numbers $(L_k)_{k\in\N}$ and $(R_k)_{k\in\N}$ such that (1) for each $k\in\N$, $L_k<R_k<L_{k+1}$; (2) $\lim_{k\rightarrow\infty}R_k-L_k=\infty$ and (3) $D\cap \bigcup_{k\in\N}[L_k,R_k]=\emptyset$. Without loss of generality we can assume that for any $k\in\N$ 
$$\sum_{s=1}^kR_s+L_{k+1}<R_{k+1}.$$
So, for any $k\geq 2^\ell-1$ and any $1\leq j_1<\cdots<j_{2^\ell-1}\leq k$,
\begin{multline*}
L_{k+1}<R_{k+1}-\sum_{s=1}^kR_s\leq\partial(R_{j_1},...,R_{j_{2^{\ell}-1}},R_{k+1})\\
\leq R_{k+1}-R_{j_{2^\ell-1}}+\sum_{s=1}^{j_{2^\ell-1}-1}R_s<R_{k+1}-L_{j_{2^\ell-1}}<R_{k+1}.
\end{multline*}
This  shows that $$D_\ell((R_k)_{k\in\N})\subseteq\bigcup_{k\in\N}[L_k,R_k].$$
Hence, $D$ is not a $\Delta_\ell^*$ set.  
\end{proof}

\section{$\Delta_\ell^*$ sets and diophantine inequalities}
As was mentioned in the introduction, we have two approaches to proving \cref{0.OddDegreeRecurrence}: an ultrafilter approach which is, so to say, soft and clean, and an elementary  approach which is based on Ramsey's theorem and which, while being more cumbersome, allows to obtain somewhat stronger finitistic results. The first approach is implemented in Subsection \ref{SubsectionUltrafilterApproach}, the second --- in Subsection \ref{SubsectionFinitisticProof}. 
\subsection{The ultrafilter approach}\label{SubsectionUltrafilterApproach}
In order to establish the $\Delta_\ell^*$ property of the set 
$$\mathcal R(v,\epsilon)=\{n\in\N\,|\,\|v(n)\|<\epsilon\}$$
we will find it convenient to work with the torus $\mathbb T=\R/\Z$ (which is identified with the unit interval $[0,1]$ with the endpoints glued up). In what follows,  while considering limits of the form $\plimgG{p}{n}{\N} v(n)$, we will think of the sequence $(v(n))_{n\in\N}$  as corresponding to the sequence $(v(n)\text{ mod 1})_{n\in\N}\in \mathbb T$. In particular
$$\plimgG{p}{n}{\N} v(n)=\alpha\text{ if and only if }\plimgG{p}{n}{\N}\|v(n)-\alpha\|=0,$$
where $\|\cdot\|$ denotes the distance to the closest integer.\\

We start with proving \cref{0.CubicCase} from the introduction (which in this section becomes \cref{2.CubicCase}). While this result forms a special case of \cref{0.OddDegreeRecurrence} (=\cref{2.OddDegreeRecurrence} below), we believe that its short proof will help the reader to better understand the underlying ideas. 
\begin{prop}\label{2.CubicCase}
For any real number  $\alpha$ and any $\epsilon>0$, the set 
$$\mathcal R(n^3\alpha,\epsilon)=\{n\in\N\,|\,\|n^3\alpha\|<\epsilon\}$$
is $\Delta_2^*$.
\end{prop}
\begin{proof}
By \cref{1.DeltaL*characterization}, it suffices to show that for any non-principal ultrafilter $p\in\beta\N$, 
$$\plimgG{p_2}{n}{\N}n^3\alpha=\plimgG{(-p_1+p_1)}{n}{\N}n^3\alpha=0.$$
As a preliminary result, we will show that for any $\gamma\in\mathbb T$ and any non-principal $p\in\beta\N$, 
$$\plimgG{p_1}{n}{\N}n\gamma=\plimgG{(-p+p)}{n}{\N}n\gamma=0.$$
Indeed, let $\plimgG{p}{n}{\N}n\gamma=\gamma_0$ and $\plimgG{(-p+p)}{n}{\N}n\gamma=\gamma_1$. Then
$$\gamma_1=\plimgG{p}{m}{\N}\plimgG{p}{n}{\N}(n-m)\gamma=\plimgG{p}{m}{\N}(\gamma_0-m\gamma)=\gamma_0-\gamma_0=0.$$
Now,  let $\plimgG{p_1}{n}{\N}n^3\alpha=\beta$
and let $\plimgG{p_1}{n}{\N}n^2\alpha=\gamma$.
We have
\begin{multline*}
\plimgG{p_2}{n}{\N}n^3\alpha=\plimgG{(-p_1+p_1)}{n}{\N}n^3\alpha=\plimgG{p_1}{m}{\N}\plimgG{p_1}{n}{\N}(n-m)^3\alpha\\
=\plimgG{p_1}{m}{\N}\plimgG{p_1}{n}{\N}(n^3\alpha-3mn^2\alpha+3m^2n\alpha-m^3\alpha)\\
=\plimgG{p_1}{m}{\N}(\beta-3m\gamma+0-m^3\alpha)
=\beta-0+0-\beta=0.
\end{multline*}
\end{proof}
We proceed now to the proof of the general case.
\begin{thm}\label{2.OddDegreeRecurrence}
For any odd real polynomial $v(x)=\sum_{j=1}^\ell a_jx^{2j-1}$ and any non-principal ultrafilter $p\in\beta\N$,
$$\plimgG{p_\ell}{n}{\N} v(n)=0.$$
Equivalently, for any $\epsilon>0$, the set 
$$\mathcal R(v,\epsilon)=\{n\in\N\,|\,\|v(n)\|<\epsilon\}$$
is $\Delta_\ell^*$.
\end{thm}
\begin{proof}
We first note that it is sufficient to prove \cref{2.OddDegreeRecurrence} for odd polynomials of the special form $v=n^{2\ell-1}\alpha$, where $\alpha\in\R$. Indeed, the general case follows, via the identity
$$\plimgG{p_\ell}{n}{\N}\sum_{j=1}^\ell a_{j}n^{2j-1}=\sum_{j=1}^\ell\plimgG{p_\ell}{n}{\N}a_jn^{2j-1},$$
from the fact that for any non-principal ultrafilter $p$, $p_\ell=(p_{\ell-t})_t$ for any $t\leq \ell$.\\
Let $p$ be a non-principal ultrafilter and let $\alpha$ be a real number. We proceed by induction on $\ell\in\N$. If $\ell=1$, then we have that
$$\plimgG{p_1}{n}{\N}n\alpha=\plimgG{p}{m}{\N}\plimgG{p}{n}{\N}(n-m)\alpha=\plimgG{p}{n}{\N}n\alpha-\plimgG{p}{m}{\N}m\alpha=0.$$
Now let $\ell>1$ and suppose that \cref{2.OddDegreeRecurrence} holds for $t<\ell$. Let $\gamma\in\mathbb T$ be such that 
$$\plimgG{p}{n}{\N}n^{2\ell-1}\alpha=\gamma.$$
Then
\begin{multline}\label{2.InductiveEquality1}
\plimgG{p_\ell}{n}{\N}n^{2\ell-1}\alpha=\plimgG{p_{\ell-1}}{m}{\N}\plimgG{p_{\ell-1}}{n}{\N}(n-m)^{2\ell-1}\alpha=\\
\plimgG{p_{\ell-1}}{m}{\N}\plimgG{p_{\ell-1}}{n}{\N}\left(n^{2\ell-1}\alpha-m^{2\ell-1}\alpha+v_m(n)\alpha+\sum_{j=1}^{\ell-1}m^{2j-1}w_j(n)\alpha\right),
\end{multline}
where for each $m\in\N$, 
$$v_m(n)=\sum_{j=1}^{\ell-1}\binom{2\ell-1}{2j}m^{2j}n^{2(\ell-j)-1}$$
and for each $j\in\{1,...,\ell-1\}$, 
$$w_j(n)=\binom{2\ell-1}{2j-1}n^{2(\ell-j)}.$$
Observe that for each $m\in\N$, $v_m(x)\alpha$ is an odd polynomial in $\R[x]$ with  $\deg(v_m(x)\alpha)<2\ell-1$. Thus, by the inductive hypothesis,
$$\plimgG{p_{\ell-1}}{m}{\N}\plimgG{p_{\ell-1}}{n}{\N}v_m(n)\alpha=0.$$
Setting $\plimgG{p_{\ell-1}}{n}{\N}w_j(n)\alpha=\alpha_j\in\mathbb T$ for each $j\in\{1,...,\ell-1\}$, it also follows from the inductive hypothesis that for each $j\in\{1,...,\ell-1\}$,
$$\plimgG{p_{\ell-1}}{m}{\N}\left(m^{2j-1}\left(\plimgG{p_{\ell-1}}{n}{\N}w_j(n)\alpha\right)\right)=\plimgG{p_{\ell-1}}{m}{\N}m^{2j-1}\alpha_j=0.$$

Note now that each of the sequences of the form $(v_m(n))_{n\in\N}$ and $(w_j(n))_{n\in\N}$ in  the right-hand side of \eqref{2.InductiveEquality1} are integer valued. It follows that each of the summands in  \eqref{2.InductiveEquality1} is well defined mod 1 (i.e. its value does not depend on our choice of representative for $\alpha\in\R/\Z$). Hence, the continuity of addition (and multiplication by a natural number) on $\R/\Z$ implies that the right-hand side of \eqref{2.InductiveEquality1} equals
\begin{multline*}
    \plimgG{p_{\ell-1}}{m}{\N}\plimgG{p_{\ell-1}}{n}{\N}n^{2\ell-1}\alpha-
    \plimgG{p_{\ell-1}}{m}{\N}\plimgG{p_{\ell-1}}{n}{\N}m^{2\ell-1}\alpha\\
    +\plimgG{p_{\ell-1}}{m}{\N}\plimgG{p_{\ell-1}}{n}{\N}v_m(n)\alpha+\sum_{j=1}^{\ell-1}(\plimgG{p_{\ell-1}}{m}{\N}m^{2j-1}\left(\plimgG{p_{\ell-1}}{n}{\N}w_j(n)\alpha\right)).
\end{multline*}
Thus, 
$$\plimgG{p_\ell}{n}{\N}n^{2\ell-1}\alpha=\gamma-\gamma+0+0=0,$$
completing the proof.
\end{proof}
\subsection{The finitistic aproach}\label{SubsectionFinitisticProof}
The finitistic approach to the proof of \cref{0.OddDegreeRecurrence} requires the use of the following version of Ramsey's Theorem. For any $k\in\N$ and any set $S$, we denote by $S^{(k)}$ the set  of all $k$-element subsets of $S$. 
\begin{thm}[Ramsey's Theorem]\label{7.FinitisticRamsey}
Let $\ell,M\in\N$ and $r\geq2^\ell$. There exists a natural number $R=R(\ell,M,r)$, with the following property:\\
For any $M$-partition 
$$\{1,...,R\}^{(2^\ell)}= \bigcup _{t=1}^M C_t,$$
one of the $C_t$'s contains $D^{(2^\ell)}$ for some set $D$ with $|D|\geq r$.
\end{thm}
\begin{rem}
Note that, in the above theorem, $\{1,...,R\}^{(2^\ell)}$ can be replaced (when convenient) by the set $\{n_1,....,n_R\}^{(2^\ell)}$, where $(n_j)_{j=1}^R$ is an R-element increasing sequence in $\N$.
\end{rem}

Let $\ell\in \N$ and $r\geq 2^\ell$. A set $E\subseteq 
\N$ is called a \textbf{$\Delta_{\ell,r}$\index{$\Delta_{\ell,r}$} set} if  there exists an r-element sequence $(n_k)_{k=1}^r$ in $\Z$ such that 
$$D_\ell((n_k)_{k=1}^r)=\{\partial(n_{j_1},...,n_{j_{2^\ell}})\,|\,1\leq j_1<\cdots<j_{2^\ell}\leq r\}\index{$D_\ell((n_k)_{k=1}^r)$}$$
is a subset of $E$.
If, in addition, $E$ contains a $\Delta_{\ell,r}$ set for each $r\geq 2^\ell$, then we say that $E$ is a \textbf{$\Delta_{\ell,0}$\index{$\Delta_{\ell,0}$} set}.\\
A set  $E\subseteq \N$ is  a \textbf{$\Delta_{\ell,r}^*$\index{$\Delta_{\ell,r}^*$} set} if it has a non-trivial intersection with any $\Delta_{\ell,r}$ set. Similarly, $E\subseteq \N$ is a \textbf{$\Delta_{\ell,0}^*$\index{$\Delta_{\ell,0}^*$} set} if it has a non-empty  intersection with any $\Delta_{\ell,0}$ set.\\\

We summarize in the following proposition the relations between the families of sets which were just introduced above. These relations follow directly from \cref{7.FinitisticRamsey}; we omit the proofs.
\begin{prop}\label{7.DeltaRSummary}
Let $\ell,\ell_1,\ell_2\in\N$, let $r\geq2^\ell$, let $r_1\geq 2^{\ell_1}$ and let $N\in\N$. The following statements hold:
\begin{enumerate}[(i)]
\item If $\ell_1\leq \ell_2$ and $2^{\ell_2-\ell_1}r_1\leq R $, then, any  $\Delta_{\ell_1,r_1}^*$ set is a $\Delta_{\ell_2,R}^*$ set.
\item A set $E\subseteq \N$ is a $\Delta_{\ell,0}^*$ set if and only if there exists $R\geq 2^\ell$ for which $E$ is a $\Delta_{\ell,R}^*$ set.
\item There exists $R\geq r$ such that for any $\Delta_{\ell,r}^*$ sets $E_1,...,E_N
\subseteq \N$, the set $E_1\cap\cdots\cap E_N$ is $\Delta_{\ell,R}^*$. In particular, $\Delta_{\ell,0}^*$ sets have the finite intersection property.
\end{enumerate}
\end{prop}
\begin{rem}
The finite intersection property of $\Delta_{\ell,0}^*$ (item (iii) in \cref{7.DeltaRSummary}) also follows from a general set-theoretical fact which states that if $\Phi$ is a partition regular family\footnote{
A family $\Phi$ of non-empty subsets of $\N$ is called partition regular if for any $N\in\N$ and any partition $\bigcup_{j=1}^NC_j\in \Phi$, at least one of $C_j$ belongs to $\Phi$.
} of non-empty subsets of $\N$, then 
$$\Phi^*=\{A\subseteq\N\,|\, \forall B\in\Phi,\, A\cap  B\neq\emptyset\}$$
has the finite intersection property. 
\end{rem}
To visualize the relations between the various classes of sets which were introduced above,  let us denote by $\Delta_\ell^*$, $\Delta_{\ell,0}^*$ and $\Delta_{\ell,r}^*$ the families of sets with the corresponding properties. Then for $\ell_1< \ell_2$ and $r_1<r_2$, we have the following diagram of equalities and  inclusions:
$$\begin{matrix}
\vdots&   &   &   &\vdots& & & &\vdots& & & &\vdots& &\vdots\\
\requal&   &   &   &\rsupseteq& & & &\rsupseteq& & & &\rsupsetneq& &\rsupsetneq\\
\Delta_{\ell_2,2^{\ell_2}}^*&\subseteq&\cdots&\subseteq&\Delta_{\ell_2,2^{\ell_2}r_1}^*&\subseteq&\cdots&\subseteq&\Delta_{\ell_2,2^{\ell_2}r_2}^*&\subseteq&\cdots&\subseteq&\Delta_{\ell_2,0}^*&\subsetneq &\Delta_{\ell_2}^*\\
\requal&   &   &   &\rsupseteq& & & &\rsupseteq& & & &\rsupsetneq& &\rsupsetneq\\
\vdots&   &   &   &\vdots& & & &\vdots& & & &\vdots& &\vdots\\
\requal&   &   &   &\rsupseteq& & & &\rsupseteq& & & &\rsupsetneq& &\rsupsetneq\\
\Delta_{\ell_1,2^{\ell_1}}^*&\subseteq&\cdots&\subseteq&\Delta_{\ell_1,2^{\ell_1}r_1}^*&\subseteq&\cdots&\subseteq&\Delta_{\ell_1,2^{\ell_1}r_2}^*&\subseteq&\cdots&\subseteq&\Delta_{\ell_1,0}^*&\subsetneq &\Delta_{\ell_1}^*\\
\requal&   &   &   &\rsupseteq& & & &\rsupseteq& & & &\rsupsetneq& &\rsupsetneq\\
\vdots&   &   &   &\vdots& & & &\vdots& & & &\vdots& &\vdots\\
\requal&   &   &   &\rsupseteq& & & &\rsupseteq& & & &\rsupsetneq& &\rsupsetneq\\
\Delta_{1,2}^*&\subseteq&\cdots&\subseteq&\Delta_{1,2r_1}^*&\subseteq&\cdots&\subseteq&\Delta_{1,2r_2}^*&\subseteq&\cdots&\subseteq&\Delta_{1,0}^*&\subsetneq &\Delta^*.
\end{matrix}
$$
(The strict inclusions appearing in the two right-most columns of the above diagram follow from \cref{2.OddPolynomialCharacterization} and    \cref{3.StrictDeltaEllinclusion}   below.)\\ 

Before embarking on the proof of the finitary version of \cref{0.OddDegreeRecurrence}, we will illustrate the main ideas in the special case $v(n)=n^3\alpha$.
\begin{prop}[Finitary version of \cref{0.CubicCase}]\label{7.CubicFinitisticCase}
For any  $\epsilon>0$, there exists $r\in\N$ such that for any real   $\alpha$, the set 
$$\mathcal R(n^3\alpha,\epsilon)=\{n\in\N\,|\,\|n^3\alpha\|<\epsilon\}$$
is $\Delta_{2,r}^*$.
\end{prop}
\begin{proof}
Let $\epsilon>0$ and let $N\in\N$ be such that $\frac{7}{N}<\epsilon$. We will show that for any  $R\in\N$ large enough, any R-element sequence $(n_j)_{j=1}^R$ in $\Z$ with $D_\ell((n_j)_{j=1}^R)\subseteq \N$ and any $\alpha\in\R$, there exist $1\leq j_1<\cdots<j_4\leq R$ such that
\begin{equation}\label{7.CubeCloseTo0}
\|(\partial(n_{j_1},...,n_{j_4}))^3\alpha\|<\epsilon.
\end{equation}
Define the sets $\mathcal Q(k_1,k_2,k_3)$ for $k_1,k_2, k_3\in\{0,...,N-1\}$ by
\begin{equation*}\label{7.SmalllCube}
\mathcal Q(k_1,k_2,k_3)=\{(\alpha_1,\alpha_2,\alpha_3)\in\mathbb T^3\,|\,\alpha_j\in \left[\frac{k_j}{N},\frac{k_j+1}{N}\right)\}
\end{equation*}
Let $R=R(2,N^3,6)$ be as in the statement of \cref{7.FinitisticRamsey}.
Let $(n_j)_{j=1}^R$ be any R-element sequence in $\Z$, let $\alpha\in\R$ and let 
$$D=\{((n_{j_2}-n_{j_1})^3\alpha,n_{j_3}(n_{j_2}-n_{j_1})^2\alpha,(n_{j_3}-n_{j_2})^2n_{j_1}\alpha)\in\mathbb T^3\,|\,1\leq j_1<\cdots<j_4\leq R\}.$$
Since
$$D=\bigcup_{k_1,k_2,k_3=0}^{N-1}\left(\mathcal Q(k_1,k_2,k_3)\cap D\right),$$
\cref{7.FinitisticRamsey} implies that there exists $1\leq t_1<\cdots <t_6\leq R$ and $0\leq s_1,s_2,s_3\leq N-1$ for which the set 
$$D'=\{((n_{t_{j_2}}-n_{t_{j_1}})^3\alpha,n_{t_{j_3}}(n_{t_{j_2}}-n_{t_{j_1}})^2\alpha,(n_{t_{j_3}}-n_{t_{j_2}})^2n_{t_{j_1}}\alpha)\in\mathbb T^3\,|\,1\leq j_1<\cdots<j_4\leq 6\}$$
is a subset of $\mathcal Q(s_1,s_2,s_3)$.\\
Since $D'\subseteq \mathcal Q(s_1,s_2,s_3)$, we have 
$$\|((n_{t_4}-n_{t_3})^3-(n_{t_2}-n_{t_1})^3)\alpha\|<\frac{1}{N},$$
$$\|3((n_{t_4}-n_{t_3})(n_{t_2}-n_{t_1})^2)\alpha\|<\frac{3}{N}$$
and 
$$\|3((n_{t_4}-n_{t_3})^2(n_{t_2}-n_{t_1}))\alpha\|<\frac{3}{N};$$
which proves \eqref{7.CubeCloseTo0}.
\end{proof}
We move now to the finitary version of \cref{0.OddDegreeRecurrence}.
\begin{thm}\label{7.FinitisticOddDegreeRecurrence}
Let $\ell\in\N$. Then for any $\epsilon>0$, there exists an $r=r(\epsilon,\ell)\geq2^\ell$  such that for any odd real polynomial $v(x)=\sum_{j=1}^{\ell}a_jx^{2j-1}$, the set $$\mathcal R(v,\epsilon)=\{n\in\N\,|\,\|v(n)\|<\epsilon\}$$ 
is a $\Delta_{\ell,r}^*$ set.
\end{thm}
\begin{proof}
By \cref{7.DeltaRSummary}, it suffices to show that for any  $\ell\in\N$ and any $\epsilon>0$ there exists an $R\in\N$ such that for any $\alpha\in\R$ the set 
$\mathcal R( n^{2\ell-1}\alpha,\epsilon)$ 
is $\Delta_{\ell,R}^*$.\\
We proceed by induction on $\ell\in\N$. The case $\ell=1$ follows from the pigeon hole principle. Now let $\ell\geq 2$ and suppose that the result holds for any $t<\ell$. Let $\epsilon>0$ and let $R_1\geq 2^{\ell-1}$  be a natural number guaranteeing that for any $\alpha_1,...,\alpha_{\ell-1}\in\R$, and any  $R_1$-element sequence $(n_k)_{k=1}^{R_1}$ in $\Z$ with $D_{\ell}((n_k)_{k=1}^{R_1})\subseteq \N$, the set
$$\{\partial(n_{k_1},...,n_{k_{2^{\ell-1}}})\,|\,1\leq k_1<\cdots<k_{2^{\ell-1}}\leq R_1\}$$
has a non-empty intersection with the set
$\mathcal R(\sum_{j=1}^{\ell-1}n^{2j-1}\alpha_j,\frac{\epsilon}{5})$.\\
By \cref{7.FinitisticRamsey}, there exists  $R\in\N$ such that for any $\alpha\in\R$ and any R-element sequence $(n_k)_{k=1}^{R}$ in $\Z$ with $D_\ell((n_k)_{k=1}^R)\subseteq \N$, there exist $\beta_1,\beta_2\in\mathbb T$ and a 2R$_1$-element subsequence $(m_s)_{s=1}^{2R_1}$ of $(n_k)_{k=1}^R$ such that for any $1\leq s_1<\cdots<s_{2^{\ell-1}}\leq R_1$ and any  $R_1+1\leq t_1<\cdots<t_{2^{\ell-1}}\leq 2R_1$, 
\begin{equation}\label{7.NearbyLargestPower}
\|(\partial(m_{t_1},...,m_{t_{2^{\ell-1}}}))^{2\ell-1}\alpha-(\partial(m_{s_1},...,m_{s_{2^{\ell-1}}}))^{2\ell-1}\alpha\|<\frac{\epsilon}{5},
\end{equation}
\begin{equation}\label{7.NearbyPolysLargeOdd} 
\|\sum_{j=1}^{\ell-1}\binom{2\ell-1}{2j}(-\partial(m_{s_1},...,m_{s_{2^{\ell-1}}}))^{2j}(\partial(m_{t_1},...,m_{t_{2^{\ell-1}}}))^{2(\ell-j)-1}\alpha
-\beta_1\|<\frac{\epsilon}{10},
\end{equation}
and 
\begin{equation}\label{7.NearbyPolysSmallOdd} 
\|\sum_{j=1}^{\ell-1}\binom{2\ell-1}{2(\ell-j)-1}(-\partial(m_{s_1},...,m_{s_{2^{\ell-1}}}))^{2(\ell-j)-1}(\partial(m_{t_1},...,m_{t_{2^{\ell-1}}}))^{2j}\alpha
-\beta_2\|<\frac{\epsilon}{10}.
\end{equation}
 By our choice of $R_1$, we have  that  for any  $1\leq s_1<\cdots<s_{2^{\ell-1}}\leq R_1$ there exists $R_1+1\leq t_1<\cdots<t_{2^{\ell-1}}\leq 2R_1$, such that
$$\|\sum_{j=1}^{\ell-1}\binom{2\ell-1}{2j}
(-\partial(m_{s_1},...,m_{s_{2^{\ell-1}}}))^{2j}(\partial(m_{t_1},...,m_{t_{2^{\ell-1}}}))^{2(\ell-j)-1}\alpha\|<\frac{\epsilon}{5},$$
which, together with \eqref{7.NearbyPolysLargeOdd}, implies that
for any  $1\leq s_1<\cdots<s_{2^{\ell-1}}\leq R_1$ and any  $R_1+1\leq t_1<\cdots<t_{2^{\ell-1}}\leq 2R_1$,
$$\|\sum_{j=1}^{\ell-1}\binom{2\ell-1}{2j}(-\partial(m_{s_1},...,m_{s_{2^{\ell-1}}}))^{2j}(\partial(m_{t_1},...,m_{t_{2^{\ell-1}}}))^{2(\ell-j)-1}\alpha\|<\frac{2\epsilon}{5}.$$
Similarly, by \eqref{7.NearbyPolysSmallOdd}, we have that for any  $1\leq s_1<\cdots<s_{2^{\ell-1}}\leq R_1$ and any  $R_1+1\leq t_1<\cdots<t_{2^{\ell-1}}\leq 2R_1$,
$$\|\sum_{j=1}^{\ell-1}\binom{2\ell-1}{2(\ell-j)-1}(-\partial(m_{s_1},...,m_{s_{2^{\ell-1}}}))^{2(\ell-j)-1}(\partial(m_{t_1},...,m_{t_{2^{\ell-1}}}))^{2j}\alpha\|<\frac{2\epsilon}{5}.$$
So, by \eqref{7.NearbyLargestPower}, we have that for any  $1\leq s_1<\cdots<s_{2^{\ell-1}}\leq R_1$ and any $R_1+1\leq t_1<\cdots<t_{2^{\ell-1}}\leq 2R_1$,
$$\|(\partial(m_{t_1},...,m_{t_{2^{\ell-1}}})-\partial(m_{s_1},...,m_{s_{2^{\ell-1}}}))^{2\ell-1}\alpha\|<\epsilon;$$
completing the proof.
\end{proof}
As we mentioned in the Introduction, for any real polynomial $v$, with $v(0)\in\Z$, the set $\mathcal R(v,\epsilon)$ is an IP$^*$ set (see \cite[Theorem 1.21]{FBook}). As a matter of fact, $\mathcal R(v,\epsilon)$ is actually an IP$^*_r$ set. Given $r\in\N$, a set $E\subseteq \N$ is called an \textbf{IP$_r$ set}\index{IP$_r$} if it contains the finite sums set of the form $$\text{FS}((n_k)_{k=1}^r)=\{n_{k_1}+\cdots+n_{k_m}\,|\,1\leq k_1<\cdots<k_m\leq r,\,1\leq m\leq r\},\index{FS$((n_k)_{k=1}^r)$}$$
where $(n_k)_{k=1}^r$ is an r-element sequence in $\N$. A set $E\subseteq \N$ is an \textbf{IP$_r^*$\index{IP$_r^*$} set} if it has a non-trivial intersection with any IP$_r$ set. It is of interest to juxtapose \cref{7.FinitisticOddDegreeRecurrence} with the following result.
\begin{thm}[see the proof of  Theorem 7.7 in \cite{BerUltraAcrossMath}]\label{3.IP0Returns} 
Let $N\in\N$. Then for any $\epsilon>0$ there exists $r=r(N,\epsilon)\in\N$  such that for any real polynomial $v(x)=\sum_{j=1}^{N}a_jx^j$ the set $\mathcal R(v,\epsilon)$
is an \text{\rm{IP$_r^*$}} set.
\end{thm}
It is natural to ask what is the relation between \cref{7.FinitisticOddDegreeRecurrence} and \cref{3.IP0Returns} when $v$ is an odd real polynomial. We will show in Section 8 that the families of sets $\Delta_{\ell,r}^*$ and IP$_r^*$ are, so to say, in general position and so \cref{7.FinitisticOddDegreeRecurrence} provides a new diophantine approximation result when $v$ is an odd real polynomial.

\section{A converse to \cref{0.OddDegreeRecurrence}}\label{SectionAConverseResult}
In this section we prove the following converse of \cref{0.OddDegreeRecurrence}.
\begin{thm}\label{2.OddPolyDegreeThm}
Let $\ell\in\N$ and let $v(x)$ be an odd real polynomial with irrational leading coefficient. If  for each $\epsilon>0$ the set
$$\mathcal R(v,\epsilon)=\{n\in\N\,|\,\|v(n)\|<\epsilon\}$$
is $\Delta_\ell^*$, then $\deg (v)\leq 2\ell-1$.
\end{thm}
We will derive \cref{2.OddPolyDegreeThm} from the following Lemma which will be also used below in Section 8. 
\begin{lem}\label{2.TechnicalLemma}
Let $p\in\beta\N$ be a non-principal ulltrafilter, let $\ell\in\N$ and let $\alpha_0,...,\alpha_\ell\in\mathbb T$ be such that for $j\in\{1,...,\ell\}$,
\begin{equation}\label{2.LemmaCondition1}
\plimgG{p}{n}{\N}n\alpha_j=0
\end{equation}
and
\begin{equation}\label{2.LemmaCondition2}
-2^{j-1}\binom{2j+1}{2}\plimgG{p}{n}{\N}n^2\alpha_j=\alpha_{j-1}.
\end{equation}
Then 
\begin{equation}\label{2.LemmaConsequence}
\plimgG{p_\ell}{n}{\N} n^{2\ell+1}\alpha_\ell=\plimgG{p}{n}{\N}n\alpha_0.
\end{equation}
\end{lem}
\begin{proof}
The proof is by  induction on $\ell\in\N$.  When $\ell=1$,  note that
$$\plimgG{(-p+p)}{n}{\N}n^3\alpha_1=\plimgG{p}{m}{\N}\plimgG{p}{n}{\N}\left((n^3-m^3)\alpha_1+3m^2n\alpha_1-3mn^2\alpha_1\right).$$
So by \eqref{2.LemmaCondition1} and \eqref{2.LemmaCondition2}, we have that
$$\plimgG{(-p+p)}{n}{\N}n^3\alpha_1=\plimgG{p}{m}{\N}m\alpha_0$$
as desired.\\
Let $\ell\geq 2$, let $p\in\beta\N$ be a non-principal ultrafilter and  let  $\alpha_0,...,\alpha_\ell\in\mathbb T$ satisfy \eqref{2.LemmaCondition1} and \eqref{2.LemmaCondition2}.  Suppose that \cref{2.TechnicalLemma} holds for all $\ell_0<\ell$.  For any $\alpha\in \mathbb T$ and any $d\in\{1,...,\ell-1\}$, let $q^{(d)}=p_{\ell-d}$ and define  $\beta^{(d)}_{d-1},...,\beta^{(d)}_0\in \mathbb T$ by letting $\beta^{(d)}_{d-1}=\alpha$ and setting
$$\beta^{(d)}_{j-1}=-2^{j-1}\binom{2j+1}{2}\plimgG{q^{(d)}}{n}{\N}n^2\beta_{j}^{(d)}$$
for each $j\in\{1,...,d-1\}$.\\
Since for any non-principal ultrafilter $q\in\beta\N$ and any $\beta\in\mathbb T$,
$$\plimgG{ q_1}{n}{\N}n\beta=\plimgG{q}{m}{\N}\plimgG{q}{n}{\N}\big((n-m)\beta\big)=0,$$
it follows from the inductive hypothesis that for any   $d\in\{1,...,\ell-1\}$, 
\begin{equation}\label{2.InductiveCorollary}
\plimgG{p_{\ell-1}}{n}{\N}n^{2d-1}\alpha=\plimgG{q^{(d)}_{d-1}}{n}{\N}n^{2d-1}\beta^{(d)}_{d-1}=\plimgG{q^{(d)}}{n}{\N}n\beta^{(d)}_0=0.
\end{equation}
Note now that for each $j\in\{2,...,\ell\}$,
\begin{multline*}
-2^{j-2}\binom{2j-1}{2}\plimgG{p}{n}{\N}\left(n^2\alpha_j\cdot 2^{j-1}\frac{(2j+1)!}{(2j-1)!}\right)\\
=-2^{j-2}\frac{(2j-1)!}{2!(2j-3)!}\plimgG{p}{n}{\N}\left(n^2\alpha_j\cdot 2^{j-1}\frac{(2j+1)!}{(2j-1)!}\right)=\alpha_{j-1}\cdot 2^{j-2}\frac{(2j-1)!}{(2j-3)!}.
\end{multline*}
Thus, the  inductive hypothesis applied to 
$$\alpha_{1}\cdot \frac{3!}{1!},...,\alpha_{\ell}\cdot 2^{\ell-1}\frac{(2\ell+1)!}{(2\ell-1)!}$$
implies that 
\begin{equation}\label{2.Inductie2}
2^{\ell-1}\frac{(2\ell+1)!}{(2\ell-1)!}\plimgG{p_{\ell-1}}{n}{\N}n^{2\ell-1}\alpha_\ell=\plimgG{p}{n}{\N}\left(n\alpha_1\cdot\frac{3!}{1!}\right)=0.
\end{equation}
So, \eqref{2.InductiveCorollary} and \eqref{2.Inductie2} imply that for any $t\in\{1,...,\ell\}$,
\begin{equation}\label{2.Lemma4(t)}
\plimgG{p_{\ell-1}}{n}{\N}n^{2(\ell-t)+1}\alpha_\ell=0.
\end{equation}
By the Binomial Theorem, the left-hand side of \eqref{2.LemmaConsequence} equals
\begin{equation}\label{2.Lemma3}
\plimgG{p_{\ell-1}}{m}{\N}\plimgG{p_{\ell-1}}{n}{\N}\sum_{j=1}^{2\ell}\binom{2\ell+1}{j}(-m)^jn^{2\ell+1-j}\alpha_\ell,
\end{equation}
which in turn  equals
\begin{multline*}
\plimgG{p_{\ell-1}}{m}{\N}\plimgG{p_{\ell-1}}{n}{\N}\sum_{s=1}^{\ell}\binom{2\ell+1}{2s-1}(-m)^{2s-1}n^{2(\ell-s)+2}\alpha_\ell\\
+\plimgG{p_{\ell-1}}{m}{\N}\plimgG{p_{\ell-1}}{n}{\N}\sum_{t=1}^{\ell}\binom{2\ell+1}{2t}(-m)^{2t}n^{2(\ell-t)+1}\alpha_\ell\\
\end{multline*}
By \eqref{2.InductiveCorollary},
$$\plimgG{p_{\ell-1}}{m}{\N}\plimgG{p_{\ell-1}}{n}{\N}\sum_{s=1}^{\ell-1}\binom{2\ell+1}{2s-1}(-m)^{2s-1}n^{2(\ell-s)+2}\alpha_\ell=0$$
and by \eqref{2.Lemma4(t)}, 
$$\plimgG{p_{\ell-1}}{m}{\N}\plimgG{p_{\ell-1}}{n}{\N}\sum_{t=1}^{\ell}\binom{2\ell+1}{2t}(-m)^{2t}n^{2(\ell-t)+1}\alpha_\ell=0.$$
So, the expression in \eqref{2.Lemma3} equals
\begin{equation*}\label{2.Lemma7}
\plimgG{p_{\ell-1}}{m}{\N}\plimgG{p_{\ell-1}}{n}{\N}\binom{2\ell+1}{2}(-1)^{2\ell-1}m^{2\ell-1}n^{2}\alpha_\ell.
\end{equation*}
Finally, by \eqref{2.LemmaCondition1}, we have that for any $m\in\N$,  
\begin{multline*}
\plimgG{p_{\ell-1}}{n}{\N}\binom{2\ell+1}{2}(-1)^{2\ell-1}m^{2\ell-1}n^{2}\alpha_\ell\\
=\plimgG{p}{n_1}{\N}\cdots\plimgG{p}{n_{2^{\ell-1}}}{\N}\binom{2\ell+1}{2}(-1)^{2\ell-1}m^{2\ell-1}(\partial(n_1,\cdots,n_{2^{\ell-1}}))^{2}\alpha_\ell\\
=\plimgG{p}{n_1}{\N}\cdots\plimgG{p}{n_{2^{\ell-1}}}{\N}\binom{2\ell+1}{2}(-1)^{2\ell-1}m^{2\ell-1}(\sum_{j=1}^{2^{\ell-1}}n_j^{2})\alpha_\ell\\
=m^{2\ell-1}\left(-2^{\ell-1}\binom{2\ell+1}{2}\plimgG{p}{n}{\N}n^2\alpha_\ell\right).
\end{multline*}
Thus, by \eqref{2.LemmaCondition2}  and the inductive hypothesis,
\begin{multline*}
\plimgG{p_\ell}{m}{\N}m^{2\ell+1}\alpha_\ell=\plimgG{p_{\ell-1}}{m}{\N}m^{2\ell-1}\left(-2^{\ell-1}\binom{2\ell+1}{2}\plimgG{p}{n}{\N}n^2\alpha_\ell\right)\\=\plimgG{p_{\ell-1}}{m}{\N}m^{2\ell-1}\alpha_{\ell-1}=\plimgG{p}{m}{\N}m\alpha_0,
\end{multline*}
 completing the proof.
\end{proof}
Now we proceed to the proof of \cref{2.OddPolyDegreeThm}.
\begin{proof}[Proof of \cref{2.OddPolyDegreeThm}]
Let $\ell\in\N$ and let $v(x)=\sum_{j=1}^{\ell'}a_jx^{2j-1}$ be an odd polynomial with irrational leading coefficient. In order to prove the contrapositive of \cref{2.OddPolyDegreeThm}, it suffices to  show that if $\ell'>\ell$, then there exists a non-principal ultrafilter $p\in \beta\N$ such that
$$\plimgG{p_\ell}{n}{\N}v(n)\neq0.$$
To prove this, suppose that $\ell'>\ell$. Choose  $\ell'$ irrational numbers $\alpha_0,...,\alpha_{\ell'-1}$ with $\alpha_{\ell'-1}=a_{\ell'}$ and with the property that $1,\alpha_0,\alpha_1,...,\alpha_{\ell'-1}$ are rationally independent. By a classical result of Hardy and Littlewood \cite{HardyLittlewood1914some}, the set
$$\{(n\alpha_0,...,n\alpha_{\ell'-1},n^2\alpha_0,...,n^2\alpha_{\ell'-1}) \,|\,n\in\N\}$$ 
is dense in $\mathbb T^{2(\ell')}$ and hence there exists an increasing sequence $(n_k)_{k\in\N}$ such that, in $\mathbb T$, 
\begin{enumerate}[(1)]
\item $\lim_{k\rightarrow\infty}n_k\alpha_0=\frac{1}{2}.$
\item For any $j\in\{1,...,\ell'-1\}$,
$$\lim_{k\rightarrow\infty}n_k\alpha_j=0.$$
\item For any $j\in\{1,...,\ell'-1\}$, $$\lim_{k\rightarrow\infty}n^2_k\left(-2^{j-1}\binom{2j+1}{2}\alpha_j\right)=\alpha_{j-1}.$$
\end{enumerate}
Let $q\in\beta\N$ be a non-principal ultrafilter with $\{n_k\,|\,k\in\N\}\in q$. By \cref{2.OddDegreeRecurrence},
$$\plimgG{q_{(\ell'-1)}}{n}{\N}v(n)=\plimgG{q_{(\ell'-1)}}{n}{\N}n^{2\ell'-1}\alpha_{\ell'-1}.$$
So, by \cref{2.TechnicalLemma}, we have 
$$\plimgG{q_{(\ell'-1)}}{n}{\N}v(n)=\plimgG{q}{n}{\N}n\alpha_0=\frac{1}{2}.$$
Finally, let $t\geq 0$ be such that $t+\ell=\ell'-1$. Letting $p=q_t$, we have 
$$q_{(\ell'-1)}=q_{(t+\ell)}=(q_t)_\ell=p_\ell.$$ 
It follows that $\plimgG{p_\ell}{n}{\N}v(n)=\frac{1}{2}$. We are done. 
\end{proof}

\section{Odd polynomials and the combinatorial properties of  sets of the form $\mathcal R(v,\epsilon)$}\label{SectionCharPol}
In this section we will show that, roughly speaking,  odd real polynomials are the only polynomials $v(x)$  such that for any $\epsilon>0$, the set  
$$\mathcal R(v,\epsilon)=\{n\in\N\,|\,\|v(n)\|<\epsilon\}$$ is $\Delta_\ell^*$ for some $\ell\in\N$. More precisely:
\begin{thm}\label{2.OddPolynomialCharacterization}
Let $\ell\in\N$ and let $v(x)$ be a real polynomial. The following are equivalent:
\begin{enumerate}[(i)]
    \item There exists a polynomial $w\in\Q[x]$ such that $w(0)\in\Z$ and $v-w$ is an odd polynomial of degree at most $2\ell-1$.
    \item For any $\epsilon>0$, there exists $r\in\N$ for which $\mathcal R(v,\epsilon)$ 
    is $\Delta_{\ell,r}^*$.
    \item For any $\epsilon>0$, $\mathcal R(v,\epsilon)$ is $\Delta_\ell^*$.
\end{enumerate}
\end{thm}
In order to prove \cref{2.OddPolynomialCharacterization} we will need the following two lemmas.
The first lemma deals with polynomials with rational coefficients and is an easy consequence of the pigeon hole principle. The second more technical lemma emphasises the distinct properties of $\mathcal R(v,\epsilon)$ for even and odd polynomials. 
\begin{lem}\label{2.Q[x]Recurrence}
Let $v(x)$ be a polynomial with rational coefficients satisfying $v(0)\in\Z$. Then there exists $r\in\N$ such that for any $\epsilon>0$,
$\mathcal R(v,\epsilon)$
is $\Delta_{1,r}^*$.
\end{lem}
\begin{proof}
Let $v(x)=\sum_{j=0}^N\frac{a_j}{b_j}x^j$, where $a_j\in\Z$, $b_j\in\N$ and $b_0=1$. Let $b=\prod_{j=1}^Nb_j$ and let $(n_k)_{k=1}^{b+1}$ be a  (b+1)-element sequence  in $\Z$. Since there exists $s,t\in\{1,...,b+1\}$ with $s<t$ for which $b|(n_t-n_s)$, we have that $v(n_t-n_s)\in\Z$.
Thus,
$$\{n\in\N\,|\,v(n)\in\Z\}$$
is $\Delta_{1,b+1}^*$.
\end{proof}

\begin{lem}\label{2.LackOfRecurrenceForEvenPowers}
Let $v(x)=\sum_{j=0}^{s}a_jx^{2j}$ be a non-zero even  polynomial such that  each $a_j$ is either  zero or irrational. Then there exists an $\epsilon>0$ such that for any  $\ell\in\N$, the set 
$\mathcal R(v,\epsilon)$
is not $\Delta_\ell^*$.
\end{lem}
\begin{proof}
It suffices to show that there exist a finite set $F\subseteq \mathbb T\setminus\{0\}$ and a non-principal ultrafilter $p\in\beta\N$ such that for any $\ell\in\N$,
\begin{equation}\label{2.EquationToProveInEvenPowers}
\plimgG{p_\ell}{n}{\N}v(n)\in F.
\end{equation}
Indeed, \eqref{2.EquationToProveInEvenPowers} implies that there is an $\epsilon>0$ with the property that for any $\ell\in\N$ the set $\mathcal R(v,\epsilon)\not\in p_\ell$. Hence, $\mathcal R(v,\epsilon)$ is not $\Delta_\ell^*$ for any $\ell\in\N$.

We now proceed to prove \eqref{2.EquationToProveInEvenPowers}. Let $s\in\N\cup\{0\}$ be such that $\deg(v)=2s$. If $s=0$, then  $v(x)=v(0)\in\R\setminus\Q$, and hence for any $\ell\in\N$ and any non-principal ultrafilter $p\in\beta\N$,
$$\plimgG{p_\ell}{n}{\N}v(n)=v(0)\neq 0.$$
To take care of the case $s\neq 0$, we invoke the fact that any vector space over $\Q$ has a basis to find irrational numbers $\beta_1,...,\beta_r$ such that (a) $1,\beta_1,...,\beta_r$ are rationally independent and (b) for  each $j\in\{1,...,s\}$,
$$a_j=q_j+b_1^{(j)}\beta_1+\cdots+b_r^{(j)}\beta_r,$$
where $b_1^{(j)},...,b_r^{(j)}\in\Z$ and $q_j\in\Q$.\\

By a result due to Hardy and Littlewood \cite{HardyLittlewood1914some}, the set
$$\{(n\beta_1,\dots,n\beta_r,n^2\beta_1,...,n^2\beta_r,...,n^{2s}\beta_1,...,n^{2s}\beta_r)\,|\,n\in\N\}$$
is dense in $\mathbb T^{(2s)r}$. Thus, we can find an increasing sequence $(n_k)_{k\in\N}$ in $\N$ such that for each $t\in\{1,...,r\}$ and $l\in\{1,...,2s-1\}$,
$$\lim_{k\rightarrow\infty}n_k^l\beta_t=0,$$
and
$$\lim_{k\rightarrow\infty}n_k^{2s}\beta_t=
\begin{cases}
\vspace{0.5em}\frac{1}{3}\frac{b_t^{(s)}}{|b_t^{(s)}|}\frac{1}{\sum_{t=1}^r|b_t^{(s)}|}\text{ if }b_{t}^{(s)}\neq 0,\\
0\text{ if }b_{t}^{(s)}=0.
\end{cases}$$
So for any $\ell\in\N$ and any  non-principal ultrafilter $p\in\beta\N$ for which $$\{n_k\,|\,k\in\N\}\in p,$$
we have 
\begin{multline*}
\plimgG{p_{\ell}}{m}{\N}v(m)=\plimgG{p}{m_1}{\N}\cdots\plimgG{p}{m_{2^\ell}}{\N}v(\partial(m_1,...,m_{2^\ell}))=\plimgG{p}{m_1}{\N}\cdots\plimgG{p}{m_{2^\ell}}{\N}\sum_{j=0}^sa_j(\partial(m_1,...,m_{2^\ell}))^{2j}\\
=\plimgG{p}{m_1}{\N}\cdots\plimgG{p}{m_{2^\ell}}{\N}(\sum_{j=1}^s\left(q_j(\partial(m_1,...,m_{2^\ell}))^{2j}+\sum_{t=1}^rb_t^{(j)}\beta_t(\partial(m_1,...,m_{2^\ell}))^{2j}\right)+a_0)\\
=\plimgG{p}{m_1}{\N}\cdots\plimgG{p}{m_{2^\ell}}{\N}(\sum_{j=1}^s\left(\sum_{t=1}^rb_t^{(j)}\beta_t(\partial(m_1,...,m_{2^\ell}))^{2j}\right)+a_0)\\
=\plimgG{p}{m_1}{\N}\cdots\plimgG{p}{m_{2^\ell}}{\N}(\sum_{h=1}^{2^\ell}\sum_{t=1}^rb_t^{(s)}\beta_t m_h^{2s})+a_0
=2^\ell\plimgG{p}{m}{\N}\left(\sum_{t=1}^rb_t^{(s)}\beta_tm^{2s}\right)+a_0\\
=2^\ell\sum_{b_t^{(s)}\neq 0}\left(b_t^{(s)}\frac{1}{3}\frac{b_t^{(s)}}{|b_t^{(s)}|}\frac{1}{\sum_{t=1}^r|b_t^{(s)}|}\right)+a_0=
\frac{2^\ell}{3}\sum_{b_t^{(s)}\neq 0}\frac{|b_t^{(s)}|}{\sum_{t=1}^r|b_t^{(s)}|}+a_0
=\frac{2^{\ell}}{3}+a_0.
\end{multline*}
Since $a_0$ is either zero or irrational, the set $\{(\frac{2^\ell}{3}+a_0)\in\mathbb T\,|\,\ell\in\N\}$ has exactly two non-zero elements. This completes the proof. 
\end{proof}
\begin{rem}\label{2.ASingleSequenceForAllL}
Using a method similar to the one used in the proof of \cref{2.LackOfRecurrenceForEvenPowers}, one can actually show that  for any  $\epsilon\in(0,\frac{1}{3})$ and any real even polynomial $v(x)$ with $v(0)=0$ and with at least one irrational coefficient, there exists an increasing sequence $(n_k)_{k\in\N}$ such that for each $\ell\in\N$, 
$$D_\ell((n_k)_{k\in\N})\subseteq\{n\in\N\,|\,\|v(n)\|>\epsilon\}.$$
\end{rem}

\begin{proof}[Proof of \cref{2.OddPolynomialCharacterization}]
(i)$\implies$(ii): Assume that there exists a polynomial $w(x)$ with rational coefficients and $w(0)\in\Z$ such that $v-w$ is a non-zero odd polynomial of degree at most $2\ell-1$ (if $v-w=0$, there is nothing to prove). Let $\epsilon>0$. By \cref{2.Q[x]Recurrence}, there exists $r_1\in\N$ such that   $\mathcal R(w,\frac{\epsilon}{2})$ is $\Delta^*_{1,r_1}$. By \cref{7.FinitisticOddDegreeRecurrence}, there 
exists $r_2\in\N$ for which the set $\mathcal R(v-w,\frac{\epsilon}{2})$ is $\Delta_{\ell,r_2}^*$. So, since
$$\mathcal R(v,\epsilon)\supseteq\mathcal R(w,\frac{\epsilon}{2})\cap\mathcal R(v-w,\frac{\epsilon}{2}),$$
\cref{7.DeltaRSummary} implies that there exists $R\in\N$ for which $\mathcal R(v,\epsilon)$ is $\Delta_{\ell,R}^*$.\\
(ii)$\implies$(iii): Note that for any $r\in\N$, a $\Delta_{\ell,r}^*$ set is a $\Delta_\ell^*$ set.\\
(iii)$\implies$(i): Let $v_e,v_o,v_r$ be real polynomials such that:
\begin{enumerate}[(1)]
\item $v(x)=v_e(x)+v_o(x)+v_r(x)$. 
\item Each of the coefficients of $v_e$ and $v_o$ are either zero or an irrational number. \item $v_e$ is an even polynomial.
\item $v_0$ is an odd polynomial. 
\item $v_r\in\Q[x]$.
\end{enumerate}
Let $\ell'\geq \ell$ be such that $\deg(v)\leq2\ell'-1$. Note that it follows from (iii) that for any non-principal ultrafilter $p\in\beta\N$,  
$$\plimgG{p_{\ell'}}{n}{\N}v(n)=0.$$
We also have, by \cref{2.OddDegreeRecurrence}, 
$$\plimgG{p_{\ell'}}{n}{\N}v_o(n)=0$$
and, by \cref{2.Q[x]Recurrence}, 
$$\plimgG{p_{\ell'}}{n}{\N}v_r(n)=0.$$
So for any non-principal $p\in\beta\N$, 
$$\plimgG{p_{\ell'}}{n}{\N}v_e(n)=0.$$
Hence, by \cref{2.LackOfRecurrenceForEvenPowers}, $v_e=0$.\\
Furthermore, by (iii) and \cref{2.Q[x]Recurrence}, we have  that  for any non-principal ultrafilter $p\in\beta\N$,
$$0=\plimgG{p_\ell}{n}{\N}v(n)=\plimgG{p_\ell}{n}{\N}v_o(n).$$
So by \cref{2.OddPolyDegreeThm}, $\deg(v_o)$ is at most $2\ell-1$.
\end{proof} 
\begin{rem}
The natural number $r$ appearing in \cref{7.FinitisticOddDegreeRecurrence}, which  guarantees that $\mathcal R(v,\epsilon)$ is a  $\Delta_{\ell,r}^*$ set for any polynomial of degree at most $2\ell-1$,  while depending on $\epsilon$ and the degree of $v$, does not depend on $v$ itself.
The situation with Theorem 5.1 is different:  the number $r$ appearing in (ii) not only depends on   $\epsilon$ and  the degree of $v$,  but also on $v$ itself. \\
To see this, let $\alpha$ be an irrational number and let $\epsilon\in(0,\frac{1}{3})$. Note that by \cref{2.ASingleSequenceForAllL}, there is  an increasing sequence $(n_k)_{k\in\N}$ in $\N$ with the property that for any $\ell\in\N$,  $D_\ell((n_k)_{k\in\N})$ does not intersect the set  $R(n^2\alpha,\epsilon)$.\\
Thus; since for each $n\in\N$, the map $x\mapsto nx$ from $\R$ to $\mathbb T$ is continuous; for each $\ell\in\N$ and each $r\geq 2^\ell$, there is an increasing $r$-element subsequence $(n_{k_j})_{j=1}^r$ of $(n_k)_{k\in\N}$ and a rational number $\frac{a}{b}$ close enough to $\alpha$, for which $D_\ell((n_{k_j})_{j=1}^r)\subseteq\N$ and 
$$D_\ell((n_{k_j})_{j=1}^r)\cap\mathcal R(n^2\frac{a}{b},\epsilon)=\emptyset.$$
So $\mathcal R(n^2\frac{a}{b},\epsilon)$ is not $\Delta_{\ell,r}^*$.
\end{rem}

\section{Applications to polynomial recurrence}\label{SecHilbert}
The goal of this section is to prove (slightly amplified versions of) Theorems \ref{0.CompactHilbert}, \ref{0.AlmostMeasurableCase} and Corollaries \ref{0.CompactMeasurable}, \ref{0.SarkozyLike} and \ref{0.WeaklyMixingCase}.\\
We start with recalling the classical Koopman-von Neumann decomposition theorem (see \cite{KoopmanVonNeumannContinuousSpectra} and Theorem 3.4, page 96, in \cite{krengelErgodicTheorems}).
\begin{thm}
Given a unitary operator  $U:\mathcal H\rightarrow \mathcal H$ one has an orthogonal decomposition
\begin{equation}\label{4.DirectSum}
\mathcal H=\mathcal H_{\text{\rm{c}}}\oplus \mathcal H_{\text{\rm{wm}}},
\end{equation}
where the $U$-invariant  (and $U^{-1}$-invariant) subspaces $\mathcal H_{\text{\rm{c}}}$  and $\mathcal H_{\text{\rm{wm}}}$ are defined as follows: 
\begin{equation}\label{4.Hc}
\mathcal H_{\text{\rm{c}}}=\overline{\langle\{f\in\mathcal H\,|\,\exists \lambda\in\mathbb T,\;Uf=e^{2\pi i\lambda}f\}\rangle}
\end{equation}
and 
\begin{equation}\label{4.Hwm}
\mathcal H_{\text{\rm{wm}}}=\{f\in\mathcal H\,|\,\lim_{N-M\rightarrow\infty}\frac{1}{N-M}\sum_{n=M+1}^N|\langle U^nf,f\rangle|=0\},
\end{equation}
\end{thm}
Throughout this section we will be using the fact that for any non-constant polynomial $v$ with $v(\Z)\subseteq\Z$ 
\begin{equation}\label{4.PolynomialWMSet}
\mathcal H_{\text{\rm{wm}}}=\{f\in\mathcal H\,|\,\forall g\in\mathcal H,\,\lim_{N-M\rightarrow\infty}\frac{1}{N-M}\sum_{n=M+1}^N|\langle U^{v(n)}f,g\rangle|=0\}.\footnote{
See formula (1.7) in \cite[page 340]{WMPet}.}
\end{equation}
\begin{thm}[Cf. \cref{0.CompactHilbert}]\label{4.CompactHilbert}
Let $U:\mathcal H\rightarrow \mathcal H$ be a unitary operator and let $v(x)=\sum_{j=1}^\ell a_jx^{2j-1}$ be a non-zero odd polynomial with $v(\Z)\subseteq\Z$. The following are equivalent:
\begin{enumerate}[(i)]
\item $U$ has discrete spectrum (i.e. $\mathcal H$ is spanned by eigenvectors of $U$).
\item For any $f\in\mathcal H$ and any $\epsilon>0$, the set 
$$\{n\in\N\,|\,\|U^{v(n)}f-f\|_{\mathcal H}<\epsilon\}$$
is $\Delta_{\ell,0}^*$.
\item For any $f\in\mathcal H$ and any $\epsilon>0$, the set 
$$\{n\in\N\,|\,\|U^{v(n)}f-f\|_{\mathcal H}<\epsilon\}$$
is $\Delta_\ell^*$.
\end{enumerate}
\end{thm}
\begin{proof}
(i)$\implies$(ii): Note that $U$ has discrete spectrum if and only if $\mathcal H=\mathcal H_{\text{\rm{c}}}$. So, by \cref{7.FinitisticOddDegreeRecurrence}, we have that for any $\epsilon>0$, there exists an $r\in\N$ such that for any $f\in\mathcal H$ and $\lambda\in \mathbb T$ with the property that $Uf=e^{2\pi i\lambda}f$, the set  
$$\{n\in\N\,|\,\|U^{v(n)}f-f\|_{\mathcal H}<\epsilon\}=\{n\in\N\,|\,\|f\|_{\mathcal H}|e^{2\pi iv(n)\lambda}-1|<\epsilon\}$$
is $\Delta_{\ell,0}^*$. Since $\Delta_{\ell,0}^*$ sets have the finite intersection property, (ii) follows.\\
(ii)$\implies$(iii): Every $\Delta_{\ell,0}^*$ set is a $\Delta_\ell^*$ set by definition.\\
(iii)$\implies$(i): Suppose, by way of contradiction, that $U$ does not have discrete spectrum. Choose $f\in\mathcal H_{\text{\rm{wm}}}$ such that $f\neq0$. Note that if  $D$ is a $\Delta_{\ell}^*$ set, then, by \cref{1.Delta*IsSyndetic}, $D$ is syndetic. Thus, we have by (iii) that 
$$\limsup_{N\rightarrow\infty}\frac{1}{N}\sum_{n=1}^N|\langle U^{v(n)}f,f\rangle|>0,$$
which contradicts \eqref{4.PolynomialWMSet}, completing the proof.
\end{proof}
\begin{cor}[Cf. \cref{0.CompactMeasurable}]\label{4.CompactMeasurePreserving}
Let $(X,\mathcal A,\mu, T)$ be an ergodic invertible probability measure preserving system. The following are equivalent:
\begin{enumerate}[(i)]
\item $(X,\mathcal A,\mu, T)$ is isomorphic to an (ergodic) translation on a compact abelian group.
\item For any odd polynomial $v(x)=\sum_{j=1}^\ell a_jx^{2j-1}$ with $v(\Z)\subseteq\Z$, any  $A\in\mathcal A$ and any $\epsilon>0$, the set 
$$\{n\in\N\,|\, \mu(A\cap T^{-v(n)}A)>\mu(A)-\epsilon\}$$
is $\Delta_{\ell,0}^*$.
\item There exists a non-zero odd polynomial $v(x)=\sum_{j=1}^\ell a_jx^{2j-1}$ with $v(\Z)\subseteq\Z$ such that for any  $A\in\mathcal A$ and any $\epsilon>0$, the set 
$$\{n\in\N\,|\, \mu(A\cap T^{-v(n)}A)>\mu(A)-\epsilon\}$$
is $\Delta_\ell^*$.
\end{enumerate}
\end{cor}
\begin{proof}
The equivalence of (i), (ii), and (iii) follows by applying  \cref{4.CompactHilbert} to the unitary operator $U_T$ induced by $T$ on $L^2(\mu)$ via the formula
$$U_Tf=f\circ T.$$
Indeed, all we need to note is that  an ergodic invertible probability measure preserving system $(X,\mathcal A,\mu,T)$ is isomorphic to a translation on a compact abelian group if and only if $L^2(\mu)=\mathcal H_{\text{\rm{c}}}$ (see \cite{neumann1932operatorenmethode} or \cite[Theorem 3.6] {waltersIntroduction}).
\end{proof}
Given $\ell\in\N$, we will say that a set $D\subseteq \N$ is  \textbf{A-$\Delta_{\ell,0}^*$}\index{A-$\Delta_{\ell,0}^*$}, (or \textbf{almost $\Delta_{\ell,0}^*$}) if there exists a set $E\subseteq \N$ with $d^*(E)=0$,
such that $D\cup E$ is $\Delta_{\ell,0}^*$. 
\begin{thm}[Cf. \cref{0.AlmostMeasurableCase}]\label{4.AlmostMeasurableCase}
Let $(X,\mathcal A,\mu,T)$ be an invertible probability measure preserving system and let  $v(x)=\sum_{j=1}^\ell a_jx^{2j-1}$ be an odd polynomial with $v(\Z)\subseteq\Z$. For any $A\in\mathcal A$ and any $\epsilon>0$,
\begin{equation}\label{4.LargeIntersectionsInLemma}
\mathcal R_A(v,\epsilon)=\{n\in\N\,|\,\mu(A\cap T^{-v(n)}A)>\mu^2(A)-\epsilon\}
\end{equation}
is A-$\Delta_{\ell,0}^*$.
\end{thm}
\begin{proof}
It suffices to show that for each $f\in L^2(\mu)$ and any $\epsilon>0$,
\begin{equation}\label{4.LargeIntersectionProof}
\{n\in\N\,|\,\langle U_T^{v(n)}f, f\rangle>\left(\int_X f\text{d}\mu\right)^2-\epsilon\}
\end{equation}
is A-$\Delta_{\ell,0}^*$.\\
Let $f\in L^2(\mu)$, $f_{\text{\rm{c}}}\in\mathcal H_{\text{\rm{c}}}$ and $f_{\text{\rm{wm}}}\in\mathcal H_{\text{\rm{wm}}}$ be such that $f=f_{\text{\rm{c}}}+f_{\text{\rm{wm}}}$. Since $\mathcal H_{\text{\rm{c}}}$ and $\mathcal H_{\text{\rm{wm}}}$ are orthogonal,   $U_T$ and $U_T^{-1}$-invariant subspaces of $L^2(\mu)$, we have that  $\langle U_T^{v(n)}f,f\rangle =\langle U_T^{v(n)}f_{\text{\rm{c}}},f_{\text{\rm{c}}}\rangle +\langle U_T^{v(n)}f_{\text{\rm{wm}}},f_{\text{\rm{wm}}}\rangle$. Now let $\epsilon>0$. Note that by \cref{4.CompactHilbert}, the set 
$$\{n\in\N\,|\,|\langle U_T^{v(n)}f_{\text{\rm{c}}},f_{\text{\rm{c}}}\rangle -\|f_{\text{\rm{c}}}\|_{L^2}^2|<\frac{\epsilon}{2}\}$$
is $\Delta_{\ell,0}^*$. Applying  Cauchy-Schwarz inequality we get
$$\|f_{\text{\rm{c}}}\|_{L^2}^2=\langle f_{\text{\rm{c}}},f_{\text{\rm{c}}}\rangle\langle \mathbbm 1_X,\mathbbm 1_X\rangle\geq |\langle f_{\text{\rm{c}}},\mathbbm 1_X\rangle|^2=\left(\int_X f_{\text{\rm{c}}}\text{d}\mu\right)^2=\left(\int_X f\text{d}\mu\right)^2,$$
which implies that the set
\begin{equation*}\label{4.CompactSetInequality}
D=\{n\in\N\,|\,\langle U_T^{v(n)}f_{\text{\rm{c}}},f_{\text{\rm{c}}}\rangle >\left(\int_X f\text{d}\mu\right)^2-\frac{\epsilon}{2}\}
\end{equation*}
is $\Delta_{\ell,0}^*$.\\
On the other hand, it follows from \eqref{4.PolynomialWMSet} that the set 
\begin{equation*}\label{4.WeakMixingComponent}
E=\{n\in\N\,|\,|\langle U_T^{v(n)}f_{\text{\rm{wm}}},f_{\text{\rm{wm}}}\rangle |\geq \frac{\epsilon}{2}\}
\end{equation*}
has zero upper Banach density. So, since for any $n\in D\setminus E$,
$$\langle U_T^{v(n)}f,f\rangle>\left( \int_X f\text{d}\mu\right)^2-\epsilon,$$
we have that
$$D\subseteq\{n\in\N\,|\,\langle U_T^{v(n)}f,f\rangle > \left(\int_Xf\text{d}\mu\right)^2-\epsilon\}\cup E.$$
Since $D$ is $\Delta_{\ell,0}^*$, the set in \eqref{4.LargeIntersectionProof} is A-$\Delta_{\ell,0}^*$. We are done.
\end{proof}
We remark that the quantity $\mu^2(A)$ in \eqref{4.LargeIntersectionsInLemma} is optimal (consider any strongly mixing system).\\

Similarly to the situation with the sets $\mathcal R(v,\epsilon)$ which was discussed in Subsection 3.2, there is an IP-flavored result dealing with the sets $\mathcal R_A(v,\epsilon)$. We need first to introduce some terminology.  A set $E\subseteq \N$ is called an \textbf{IP$_0$\index{IP$_0$} set} if it is an IP$_r$ set for each $r\in\N$. The set $E\subseteq \N$ is an \textbf{IP$_0^*$\index{IP$_0^*$} set} if it has a non-trivial intersection with any  IP$_0$ set. Finally,  a set $D\subseteq \N$ is called  \textbf{A-IP$_0^*$}\index{A-IP$_0^*$} (or almost IP$_0^*$) if there exists a set $E\subseteq \N$ with $d^*(E)=0$ such that $D\cup E$ is IP$_0^*$.
\begin{thm}[Cf. \cite{AlmostIPBerLeib}, Theorem 1.8, case $k=1$]\label{3.AlmostIPRecurrence}\
Let $(X,\mathcal A,\mu,T)$ be an invertible probability measure preserving system and let $v(x)=\sum_{j=1}^Na_jx^j$ be a polynomial with $v(\Z)\subseteq\Z$. For any $\epsilon>0$, the set
$$\mathcal R_A(v,\epsilon)=\{n\in\N\,|\,\mu(A\cap T^{-v(n)}A)>\mu^2(A)-\epsilon\}$$
is \rm{A-IP$^*_0$}.
\end{thm} 
We will show in Section 8 below (see  \cref{3.FirstImportantApplicationA}) that for each $\ell\in\N$, there exists an \rm{$\text{A-IP}_0^*$} set which is not $\text{A-}\Delta_{\ell,0}^*$. Thus, \cref{4.AlmostMeasurableCase} provides new information about sets of the form $\mathcal R_A(v,\epsilon)$.\\

We will give now two corollaries of \cref{4.AlmostMeasurableCase}. The first one is a variant of the Furstenberg-S{\'a}rk{\" o}zy theorem (see \cite{sarkozy1978difference} and \cite[Theorem 3.16]{FBook}). The second establishes a new recurrence property for weakly mixing systems.
\begin{cor}[Cf. \cref{0.SarkozyLike}]\label{4.SarkozyREsult}
Let $E\subseteq\N$ and assume that $d^*(E)>0$. Then for any odd polynomial  $v(x)=\sum_{j=1}^\ell a_jx^{2j-1}$  with $v(\Z)\subseteq\Z$, the set
$$\{n\in\N\,|\,v(n)\in E-E\}$$
is A-$\Delta_{\ell,0}^*$.
\end{cor}
\begin{proof}
By Furstenberg's correspondence principle (see \cite[Theorem 1.1]{BerErgodicTheory1985}), there exists an  invertible probability measure preserving system $(X,\mathcal A,\mu, T)$ and a set $A\in\mathcal A$ with $\mu(A)=d^*(E)$ such that for all $n\in\Z$,\
$$d^*(E\cap (E-n))\geq \mu(A\cap T^{-n}A).$$
\cref{4.AlmostMeasurableCase} implies that the set 
$$D=\{n\in\N\,|\,d^*(E\cap (E-v(n)))>0\}$$
is A-$\Delta_{\ell,0}^*$. Since
$$D\subseteq\{n\in\N\,|\, v(n)\in E-E\},$$
we are done. 
\end{proof}
\begin{rem}
Actually, Furstenberg's correspondence principle allows us to get a finer result. Namely, given any $\epsilon>0$, any odd  polynomial $v(x)=\sum_{j=1}^\ell a_jx^{2j-1}$ with $v(\Z)\subseteq\Z$ and any set $E\subseteq \N$ with $d^*(E)>0$,
$$\{n\in\N\,|\,d^*(E\cap (E-v(n)))>(d^*(E))^2-\epsilon\}$$
is A-$\Delta_{\ell,0}^*$.
\end{rem}
\begin{cor}[Cf. \cref{0.WeaklyMixingCase}]\label{4.WeaklyMixingCase}
Let $v(x)=\sum_{j=1}^\ell a_jx^{2j-1}$ be a non-zero odd polynomial with $v(\Z)\subseteq\Z$ and let $(X,\mathcal A,\mu, T)$ be an invertible probability measure preserving system. The following are equivalent:
\begin{enumerate}[(i)]
\item $T$ is weakly mixing.
\item For any $A,B\in\mathcal A$ and any $\epsilon>0$,
$$\mathcal R_{A,B}(v,\epsilon)=\{n\in\N\,|\,|\mu(A\cap T^{-v(n)}B)-\mu(A)\mu(B)|<\epsilon\}$$
is A-$\Delta_{\ell,0}^*$.
\item For any $A,B\in\mathcal A$ and any $\epsilon>0$,
$$\mathcal R_{A,B}(v,\epsilon)=\{n\in\N\,|\,|\mu(A\cap T^{-v(n)}B)-\mu(A)\mu(B)|<\epsilon\}$$
is A-$\Delta_{\ell}^*$.
\end{enumerate}
\end{cor}
\begin{proof}
(i)$\implies$(ii): Since $(X,\mathcal A,\mu, T)$ is weakly mixing, we have that 
$$L^2(\mu)=\mathbb C \mathbbm 1_X\oplus \mathcal H_{\text{\rm{wm}}}.$$ 
Hence, it follows from \eqref{4.PolynomialWMSet} that for any $\epsilon>0$ and any $f,g\in L^2(\mu)$, the set
$$E=\{n\in\N\,|\,|\langle U_T^{v(n)}f,g\rangle-\int_Xf\text{d}\mu\int_Xg\text{d}\mu|\geq \epsilon\}$$
satisfies $d^*(E)=0$, which, in turn, implies that 
$$\{n\in\N\,|\,|\langle U_T^{v(n)}f,g\rangle-\int_Xf\text{d}\mu\int_Xg\text{d}\mu|< \epsilon\}$$
is A-$\Delta_{\ell,0}^*$.\\
(ii)$\implies$(iii): Every A-$\Delta_{\ell,0}^*$ set is an  A-$\Delta_\ell^*$ set.\\
(iii)$\implies$(i): Suppose that  $(X,\mathcal A,\mu,T)$ is not weakly mixing. It suffices to show that for some $A\in\mathcal A$ and some $\epsilon>0$, the set
$$E_{\epsilon}(A)=\{n\in\N\,|\,|\mu(A\cap T^{-v(n)}A)-\mu^2(A)|<\epsilon\}$$
is not A-$\Delta_\ell^*$.\\
Since $T$ is not weakly mixing, there exists a non-constant function $f\in\mathcal H_{\text{\rm{c}}}$. Noting  that either the real or imaginary part of $f$ is non-constant, we assume without loss of generality that $f$ is real valued. Observe that we can find $a,c\in\R$, $a<c$, such that $\mu(f^{-1}([a,c)))\in(0,1)$. Furthermore, since $[a,c)$ is the union of closed intervals of the form $[a,c']$, $c'\in[a,c)$, we can find  $b\in [a,c)$ for which $\mu(f^{-1}([a,b]))\in(0,1)$. Let $A=f^{-1}([a,b])$.\\
For each $\delta>0$, let $A_\delta=f^{-1}([a-\delta,b+\delta])$. Note that for any $\gamma>0$ and any $\epsilon>0$ there exists an $r=r(\epsilon,\gamma)>0$ such that for any $n\in\{m\in\N\,|\,\|U_T^{v(m)}f-f\|_{L^2}<r\}$,
$$\mu(A)-\frac{\epsilon}{2}\leq \mu(A\cap \{x\in A\,|\,|U_T^{v(n)}f(x)-f(x)|<\gamma\}).$$
Since $\{x\in A\,|\,|U_T^{v(n)}f(x)-f(x)|<\gamma\}\subseteq T^{-v(n)}A_\gamma$, we have $\mu(A)-\frac{\epsilon}{2}\leq \mu(A\cap T^{-v(n)}A_\gamma)$. Noting that for any 
$\epsilon>0$ there exists a $\delta>0$ such that $\mu(A_\delta\setminus A)<\frac{\epsilon}{2}$, we obtain that for $r$ small enough and $n\in \{m\in\N\,|\,\|U_T^{v(m)}f-f\|_{L^2}<r\}$,
$$\mu(A)-\epsilon\leq \mu(A\cap T^{-v(n)}A_\delta)-\frac{\epsilon}{2}< \mu(A\cap T^{-v(n)}A).$$
By \cref{4.CompactHilbert}, (iii), for any $r>0$ the set $\{m\in\N\,|\,\|U_T^{v(m)}f-f\|_{L^2}<r\}$ is $\Delta_\ell^*$. It follows that for any $\epsilon>0$ the set
$$D_\epsilon=\{n\in\N\,|\,\mu(A\cap T^{-v(n)}A)>\mu(A)-\epsilon\}$$
is $\Delta_{\ell}^*$. Since $\mu(A)>\mu^2(A)$, we can find an $\epsilon>0$ such that for any $n\in D_\epsilon$,
$$\mu(A\cap T^{-v(n)}A)>\mu^2(A)+\epsilon.$$
It follows that given $E\subseteq\N$ with $d^*(E)=0$, $(E_{\epsilon}(A)\cup E)\cap D_\epsilon=E\cap D_\epsilon$ and hence 
$$d^*((E_{\epsilon}(A)\cup E)\cap D_\epsilon)=0.$$ 
By noting that the intersection of any two $\Delta_\ell^*$ sets is again $\Delta_\ell^
*$ and that for any $\Delta_\ell^*$ set $D$, $d^*(D)>0$, we conclude that $E_\epsilon(A)\cup E$ is not $\Delta_\ell^*$. Since $E$ was arbitrary,  $E_\epsilon(A)$ is not A-$\Delta_\ell^*$, completing the proof. 
\end{proof}
\begin{rem}
In the next section we will show that in the statement of \cref{4.WeaklyMixingCase}, A-$\Delta_{\ell,0}^*$ and A-$\Delta_\ell^*$ can not be replaced by $\Delta_\ell^*$ (see \cref{5.MainResultOfSection7} below).
\end{rem}

\section{\cref{0.WeaklyMixingCase} cannot be improved}\label{SecExample}
Our goal in this section is to prove the following result:
\begin{prop}\label{5.ExampleProposition}
For any odd polynomial $v(x)$ with $v(\Z)\subseteq\Z$, there exists a weakly mixing invertible probability measure preserving system $(X,\mathcal A,\mu, T)$, a set $A\in\mathcal A$ with $\mu(A)\in(0,1)$ and a non-principal ultrafilter $p\in\beta\N$ such that for any $\ell\in\N$,
\begin{equation}\label{5.MainObjective}
\plimgG{p_\ell}{n}{\N}\mu(A\cap T^{-v(n)}A)=\mu(A).
\end{equation}
\end{prop}
\begin{rem}\label{5.MainResultOfSection7}
Let the invertible probability measure preserving system $(X,\mathcal A, \mu, T)$ and the set $A\in\mathcal A$ be as in the statement of  \cref{5.ExampleProposition}. Then for any small enough $\epsilon>0$,
$$\mathcal R_A(v,\epsilon)=\{n\in\N\,|\,|\mu(A\cap T^{-v(n)}A)-\mu^2(A)|<\epsilon\}$$
is not $\Delta_{\ell}^*$ for any $\ell\in\N$. In particular, in the statement of \cref{4.WeaklyMixingCase}, A-$\Delta_\ell^*$ and  A-$\Delta_{\ell,0}^*$ can not be replaced by $\Delta_{\ell}^*$ (or $\Delta_{\ell,0}^*$).
\end{rem}
In the proof of \cref{5.ExampleProposition}, we will be using the fact that for any continuous symmetric probability measure $\gamma$ on $\mathbb T$,\footnote{
A Borel probability measure $\gamma$ on $\mathbb T$ is called symmetric if for any $n\in\Z$
$$\int_\mathbb T e^{2\pi i nx}\text{d}\gamma(x)=\int_\mathbb T e^{-2\pi inx}\text{d}\gamma(x).$$
}
there exists a weakly mixing invertible probability measure preserving system $(X,\mathcal A, \mu, T)$ called a Gaussian system. Such a system has the property that for some  $f\in L^2(\mu)$:
\begin{enumerate}[(1)]
\item  For any Borel-measurable $B\subseteq \R$,
\begin{equation}\label{5.GaussianDistributioon}
\mu(f^{-1}(B))=\frac{1}{\sqrt{2\pi}}\int_Be^{-\frac{x^2}{2}}\text{d}x
\end{equation}
(i.e. $f$ has a Gaussian distribution with mean 0 and variance 1).
\item For any $n\in\Z$,
\begin{equation}\label{5.InnerProductOff}
\langle U_T^nf,f\rangle=\int_\mathbb T e^{2\pi inx}\text d\gamma(x).
\end{equation}
\end{enumerate}
Note that for such a function $f$, $\|f\|_{L^2}=1$.
For information on Gaussian systems see for example \cite[Chapters 8.2 and 14]{cornfeld1982ergodic}.\\
\begin{proof}[Proof of \cref{5.ExampleProposition}] 
Let $v(x)$ be an odd polynomial with $v(\Z)\subseteq \Z$. It is not hard to check that $v(x)\in\Q[x]$ and hence, there exists an $m\in\N$ for which $mv(x)\in\Z[x]$. Suppose that \cref{5.ExampleProposition} holds for the odd polynomial $mv(x)$. Then, there exists a weakly mixing invertible measure preserving transformation $T$ and a set $A\in\mathcal A$ satisfying \eqref{5.MainObjective}. By considering the weakly mixing transformation $T^m$, one sees that \cref{5.ExampleProposition} also holds for $v(x)\in\Q[x]$. Thus, without loss of generality, we can assume that $v(x)\in\Z[x]$.
For convenience, we will prove  \cref{5.ExampleProposition} for $v(x)=x^3$. The proof for a general odd polynomial with integer coefficients can be done similarly.\\
Note that it is enough to show that  there exists a continuous symmetric probability measure $\gamma$ on $\mathbb T$  with the property that for some non-principal ultrafilter $p\in\beta\N$ and any $\ell\in\N$,
\begin{equation}\label{5.ConditionWannaShow}
\plimgG{p_\ell}{n}{\N}\int_\mathbb T e^{2\pi i n^3 x}\text d\gamma(x)=1.
\end{equation}
Indeed, if such a probability measure $\gamma$ exists, we would be able to find a Gaussian system $(X,\mathcal A,\mu, T)$ and a function $f\in L^2(\mu)$ satisfying  \eqref{5.GaussianDistributioon} and \eqref{5.InnerProductOff}. For such a function and each $\ell\in\N$ we will have that 
$$\plimgG{p_\ell}{n}{\N}\langle U_T^{n^3}f,f\rangle=\plimgG{p_\ell}{n}{\N}\int_\mathbb T e^{2\pi i n^3 x}\text d\gamma(x)=1=\|f\|_{L^2}^2.$$
So,
$$\plimgG{p_\ell}{n}{\N}U_T^{n^3}f=f$$
in the norm-topology of $L^2(\mu)$.
Thus, for $A=f^{-1}([-1,1])$, we will have 
$$\plimgG{p_\ell}{n}{\N}\mu(A\cap T^{-n^3}A)=\mu(A),$$
which proves \eqref{5.MainObjective}.\\
Note also that, in order to achieve our goal, it is enough to find a not necessarily symmetric continuous Borel probability measure $\rho$ on $\mathbb T$ such that for some non-principal ultrafilter $p\in\beta\N$ and any $\ell\in\N$, 
\begin{equation}\label{5.TrueConditionToFind}
\plimgG{p_\ell}{n}{\N}\int_\mathbb T e^{2\pi i n^3 x}\text d\rho(x)=1.
\end{equation}
Indeed, let $\rho$ be such a measure. Define  $\tilde\rho$ to be the unique probability measure satisfying
$$\int_\mathbb T e^{2\pi i nx}\text{d}\tilde\rho(x)=\int_{\mathbb T}e^{-2\pi i nx}\text{d}\rho(x)$$
for each $n\in\Z$. Then, the measure $\gamma=\frac{\rho+\tilde\rho}{2}$ is a symmetric continuous Borel probability measure on $\mathbb T$ for which \eqref{5.ConditionWannaShow} holds.\\

Let $\mathcal C=\{0,1\}^\N$ be  endowed with the product topology. Let  $\nu$ be the $(\frac{1}{2},\frac{1}{2})$-probability measure on $\mathcal C$. We will introduce now  a continuous function $F:\mathcal C\rightarrow \mathbb T$ such that the measure $\rho=\nu\circ F^{-1}$ is a continuous Borel probability measure on $\mathbb T$  satisfying \eqref{5.TrueConditionToFind}.\\
For each $k\in\N$, let $n_k=2^{6^k}$ and let $F:\mathcal C\rightarrow \mathbb T$ be defined by 
$$F(\omega)=\sum_{s\in\N}\frac{\omega(s)}{n_s^3},\text{ }\omega\in\mathcal C.$$
Note that $F$ is continuous, injective and for each $\omega\in\mathcal C$,
\begin{multline*}
\limsup_{k\rightarrow\infty}\|n^3_kF(\omega)\|=\limsup_{k\rightarrow\infty}\|n^3_k\sum_{s\in\N}\frac{\omega(s)}{n^3_s}\|=\limsup_{k\rightarrow\infty}\|\sum_{s\in\N}\frac{n^3_k}{n^3_s}\omega(s)\| \\
=\limsup_{k\rightarrow\infty}\|\sum_{s=1}^{k-1}\frac{2^{3\cdot 6^k}}{2^{3\cdot6^s}}\omega(s)+\omega(k)+\sum_{s=k+1}^\infty\frac{2^{3\cdot6^k}}{2^{3\cdot6^s}}\|\\
\leq\limsup_{k\rightarrow\infty}\left(\|\sum_{s=1}^{k-1}2^{3(6^k-6^s)}\omega(s)\|+\|\omega(k)\|+\|\sum_{s=k+1}^\infty\frac{\omega(s)}{2^{3(6^s-6^k)}}\|\right).
\end{multline*}
So, since 
$$\sum_{s=1}^{k-1}2^{3(6^k-6^s)}\omega(s)\in \Z$$ and
$$|\sum_{s=k+1}^\infty\frac{\omega(s)}{2^{3(6^s-6^k)}}|\leq \frac{1}{2^{6^k}}\sum_{s\in\N}\frac{1}{2^s}<\frac{1}{2^{6^k}},$$ 
we have that $\lim_{k\rightarrow\infty}\|n_k^3F(\omega)\|=0$.\\
We also have that for any $M\in\Z$ and any $\omega\in\mathcal C$,
\begin{multline*}
    \limsup_{k\rightarrow\infty}\|Mn^2_kF(\omega)\|=\limsup_{k\rightarrow\infty}\|M\sum_{s\in\N}\frac{2^{2\cdot 6^k}}{2^{3\cdot6^s}}\omega(s)\| \\
    \leq|M|\limsup_{k\rightarrow\infty}\left(\|\sum_{s=1}^{k-1}2^{2\cdot6^k-3\cdot 6^s}\omega(s)\|+\|\frac{\omega(k)}{2^{6^k}}\|+\|\sum_{s=k+1}^\infty\frac{\omega(s)}{2^{3\cdot 6^s- 2\cdot 6^k}}\|\right)=0,
\end{multline*}
which implies that $\lim_{k\rightarrow \infty}\|Mn_k^2F(\omega)\|=0$.\\ 
Similarly, for any $M\in\Z$ and any $\omega \in \mathcal C$,
$$\lim_{k\rightarrow\infty}\|Mn_kF(\omega)\|=0.$$
Thus, for any continuous function $g:\mathbb T\rightarrow\mathbb C$ and any $M_1,M_2\in\Z$,
\begin{equation}\label{5.SequentialLimit}
\lim_{k\rightarrow\infty}\int_\mathcal C g(F(\omega))e^{2\pi i(n_k^3+M_1n_k^2+M_2n_k)F(\omega)}\text{d}\nu(\omega)=\int_\mathcal C g(F(\omega))\text{d}\nu(\omega).
\end{equation}

Now let $p\in\beta\N$ be a non-principal ultrafilter with $\{n_k\,|\,k\in\N\}\in p$. We will show that for any $\ell\in\N$, any $M_1,M_2\in\Z$ and any continuous function $g:\mathbb T\rightarrow \mathbb C$, we have
\begin{equation}\label{5.StrongerResultInduction}
\plimgG{p_\ell}{n}{\N}\int_\mathcal C g(F(\omega))e^{2\pi i(n^3+M_1n^2+M_2n)F(\omega)}\text{d}\nu(\omega)=\int_\mathcal C g(F(\omega))\text{d}\nu(\omega).
\end{equation}
We proceed by induction on $\ell\in\N\cup\{0\}$. When $\ell=0$, $p_0=p$ and,  by \eqref{5.SequentialLimit}, we have that
$$\plimgG{p}{n}{\N}\int_{\mathcal C}g(F(\omega))e^{2\pi i(n^3+M_1n^2+M_2n)F(\omega)}\text{d}\nu(\omega)=\int_\mathcal Cg(F(\omega))\text{d}\nu(\omega).$$
Next, fix $\ell\in\N$ and suppose that \eqref{5.StrongerResultInduction} holds for $\ell'<\ell$. Then, by \cref{1.LemaDeltaEquality}, we have that the left-hand side of \eqref{5.StrongerResultInduction} equals the expression
\begin{equation*}\label{5.FirstStepInductiveStep}
\plimgG{p_{\ell-1}}{m}{\N}\plimgG{p_{\ell-1}}{n}{\N}\int_\mathcal C g(F(\omega))e^{2\pi i((n-m)^3+M_1(n-m)^2+M_2(n-m))F(\omega)}\text{d}\nu(\omega),
\end{equation*}
which  in turn equals 
\begin{multline}\label{5.SecondStepInductiveStep}
\plimgG{p_{\ell-1}}{m}{\N}\plimgG{p_{\ell-1}}{n}{\N}\int_\mathcal C g(F(\omega))e^{-2\pi i(m^3-M_1m^2+M_2m)F(\omega)}\\
e^{2\pi i(n^3-(3m-M_1)n^2+(3m^2-2mM_1+M_2)n)F(\omega)}\text{d}\nu(\omega)
\end{multline}
By applying the inductive hypothesis to the function
$$G_m(x)=g(x)e^{-2\pi i(m^3-M_1m^2+M_2m)x},$$
we have that \eqref{5.SecondStepInductiveStep} equals
$$\plimgG{p_{\ell-1}}{m}{\N}\int_\mathcal C G_m(F(\omega))\text{d}\nu(\omega).$$
It follows that
\begin{multline*}
\plimgG{p_\ell}{n}{\N}\int_\mathcal C g(F(\omega))e^{2\pi i(n^3+M_1n^2+M_2n)F(\omega)}\text{d}\nu(\omega)\\
=\plimgG{p_{\ell-1}}{m}{\N}\int_\mathcal C g(F(\omega))e^{-2\pi i(m^3-M_1m^2+M_2m)F(\omega)}\text{d}\nu(\omega)\\
=\plimgG{p_{\ell-1}}{m}{\N}\overline{\int_\mathcal C \overline{g(F(\omega))}e^{2\pi i(m^3-M_1m^2+M_2m)F(\omega)}\text{d}\nu(\omega)}
=\int_\mathcal C g(F(\omega))\text{d}\nu(\omega),
\end{multline*}
completing the induction.\\
Finally, since $\rho=\nu\circ F^{-1}$, we have that for any $\ell\in\N$,
$$\plimgG{p_\ell}{m}{\N}\int_\mathbb T e^{2\pi i m^3x}\text {d}\rho(x)=\plimgG{p_\ell}{m}{\N}\int_\mathcal C e^{2\pi im^3F(\omega)}\text{d}\nu(\omega)=1,$$
showing that \eqref{5.TrueConditionToFind} holds for any non-principal ultrafilter $p\in\beta\N$ for which 
$$\{n_k\,|\,k\in\N\}\in p.$$
\end{proof}
\section{Hierarchy of notions of largeness}\label{SectionNotionsOfLargness}
In this section we will review the relations between various  notions of largeness which played an instrumental role in the formulations and proofs of the results concerning the sets $\mathcal R(v,\epsilon)$ and $\mathcal R_A(v,\epsilon)$. 
In particular, we will supply the proofs of the results mentioned in  Subsection 3.2 and Section 6 which juxtapose the $\Delta^*$-flavored theorems \ref{7.FinitisticOddDegreeRecurrence}  and \ref{4.AlmostMeasurableCase} with the \rm{IP$^*$}-flavored theorems \ref{3.IP0Returns} and \ref{3.AlmostIPRecurrence} (see \cref{3.FirstImportantApplicationA}  and \cref{3.SecondImportantApplication} below). 
\subsection{Some classes of subsets of $\N$}
In this subsection we review the definitions and properties of the families of sets (such as, say $\Delta_{\ell,r}$) which appeared before in this paper and which were employed in the formulations and proofs of various results dealing with diophantine approximation and recurrence. The material 
presented in this subsection will facilitate the discussion in Subsection 8.2, where the relations between these families of sets are discussed and summarised.\\ 
The following table presents in a compact form the pertinent definitions.\\
\begin{center}
 \begin{tabular}{| M{2cm} | M{2.5cm} |m{7cm} |} 
 \hline
 \textbf{Family} & \textbf{Parameters} & \textbf{Each member contains a set of the form...}  \\  
 \hline\hline
   $\Delta_{\ell,r}$ & $\ell\in\N$, $r\geq 2^\ell$ & $D_\ell((n_k)_{k=1}^r)$, where $(n_k)_{k=1}^r$ is an r-element sequence in $\Z$. \\ 
 \hline
    $\Delta_{\ell,0}$ & $\ell\in\N$ & $D_\ell((n_k)_{k=1}^r)$ for each $r\geq 2^\ell$. Here, for each $r\geq 2^\ell$, $(n_k)_{k=1}^r$ is an r-element sequence in $\Z$. \\ 
 \hline
     $\Delta_{\ell}$ & $\ell\in\N$ & $D_\ell((n_k)_{k\in\N})$, where $(n_k)_{k\in\N}$ is an increasing sequence in $\N$. \\ 
\hline
  $\Delta_{\ell,0}$-rich\index{$\Delta_{\ell,0}$-rich} & $\ell\in\N$ & A set $D\subseteq\N$ such that for any $E\subseteq \N$ with $d^*(E)=0$, $D\setminus E$ is a $\Delta_{\ell,0}$ set \\ 
\hline
  $\Delta_\ell$-rich\index{$\Delta_\ell$-rich} & $\ell\in\N$ & A set $D\subseteq\N$ such that for any $E\subseteq \N$ with $d^*(E)=0$, $D\setminus E$ is a $\Delta_\ell$ set \\
\hline
 IP$_r$& $r\in\N$ & $\text{FS}((n_k)_{k=1}^r)$, where $(n_k)_{k=1}^r$ is an r-element sequence in $\N$.\\
 \hline
  IP$_0$& -- & $\text{FS}((n_k)_{k=1}^r)$ for each $r\in\N$. Here, for each $r\in\N$, $(n_k)_{k=1}^r$ is an r-element sequence in $\N$.\\
  \hline
  IP\index{IP}& -- & $\text{FS}((n_k)_{k\in\N})$ for some increasing sequence $(n_k)_{k\in\N}$  in $\N$.\\
   \hline
  IP$_0$-rich\index{IP$_0$-rich}& -- &  A set $\Gamma\subseteq \N$  such that for any $E\subseteq\N$ with $d^*(E)=0$, $\Gamma\setminus E$ is an IP$_0$ set.\\
     \hline
  IP-rich\index{IP-rich}& -- &  A set $\Gamma\subseteq \N$  such that for any $E\subseteq\N$ with $d^*(E)=0$, $\Gamma\setminus E$ is an IP set.\\
  \hline
\end{tabular}
\end{center}
If $\Phi$ is a family of subsets of $\N$,  $\Phi^*$ customarily stands for the family of subsets of $\N$ having a non-trivial intersection with any member of $\Phi$ (see for example \cite[Chapter 9, Section 1]{FBook}). For instance, IP$^*$ denotes the family of all subsets of $\N$ having a non-trivial intersection with any IP set. The following lemma provides useful characterizations for the families $(\Delta_\ell\text{-rich})^*$,  $(\Delta_{\ell,0}\text{-rich})^*$, (IP-rich)$^*$, and (IP$_0$-rich)$^*$.\footnote{The families $(\Delta_{\ell,0}\text{-rich})^*$, $(\Delta_{\ell,0}\text{-rich})^*$, and (IP$_0$-rich)$^*$ were introduced in the previous sections using a different notation. See \cref{3.RemarkAfterNotationalLemma}.} 
\begin{lem}\label{3.AlmostAnd*}
Let $D\subseteq \N$ and let $\ell\in\N$. Then:
\begin{enumerate}[(i)]
    \item $D$ is \text{\rm{$(\Delta_{\ell,0}\text{-rich})^*$}} if and only if there exists a set $E\subseteq\N$ with $d^*(E)=0$ such that $D\cup E$ is a $\Delta_\ell^*$ set.\index{A-$\Delta_\ell^*$} 
    \item $D$ is \text{\rm{$(\Delta_{\ell,0}\text{-rich})^*$}} if and only if there exists a set $E\subseteq\N$ with $d^*(E)=0$ such that $D\cup E$ is a $\Delta_{\ell,0}^*$ set.\index{A-$\Delta_{\ell,0}^*$} 
    \item (See \cite{mccutcheonDsetRich}) $D$ is \text{\rm{(IP-rich)$^*$}} if and only if there exists a set $E\subseteq\N$ with $d^*(E)=0$ such that $D\cup E$ is an \text{\rm{IP$^*$}} set.
    \item $D$ is \text{\rm{(IP$_0$-rich)$^*$}} if and only if there exists a set $E\subseteq\N$ with $d^*(E)=0$ such that $D\cup E$ is an \text{\rm{IP$_0^*$}} set. \index{A-IP$_0^*$}
\end{enumerate}
\end{lem}
\begin{proof}
Statements (i) through (iv) follow from the following general fact about any family $\Phi$ of subsets of $\N$ with the property that if $A\in \Phi$, $B\subseteq \N$, and $A\subseteq B$, $B\in\Phi$. We will say that a set $F\subseteq\N$ is $\Phi$-rich if for any $E\subseteq \N$ with $d^*(E)=0$, $F\setminus E\in \Phi$.\\
\begin{adjustwidth}{0.5cm}{0.5cm}
Let $D\subseteq \N$. $D$ is ($\Phi$-rich)$^*$ if and only if there exists an $E\subseteq \N$ with $d^*(E)=0$ such that $D\cup E$ is a $\Phi^*$ set.\\
\end{adjustwidth}
First, suppose that there exists a set $E\subseteq \N$ with $d^*(E)=0$ such that  $D\cup E$ is $\Phi^*$. Let $S$ be a $\Phi$-rich set. Since $S\setminus E$ is a $\Phi$ set, we have that
$$\emptyset\neq (D\cup E)\cap (S\setminus E)\subseteq D\cap S.$$
This shows that $D$ has a non-trivial intersection with every $\Phi$-rich set and hence it is   ($\Phi$-rich)$^*$.\\
For the other direction, suppose that for any $E\subseteq \N$ with $d^*(E)=0$, $D\cup E$ is not $\Phi^*$. This implies that for any given $E\subseteq \N$ with $d^*(E)=0$, $B=\N\setminus (D\cup E)$ contains an $A\in \Phi$ and hence $B\in \Phi$. Noting that for any $E\subseteq \N$, $(\N\setminus D)\setminus E=\N\setminus(D\cup E)$, we see that $\N\setminus D$ is a $\Phi$-rich set. So, $D$  is not ($\Phi$-rich)$^*$.
\end{proof}
\begin{rem}\label{3.RemarkAfterNotationalLemma}
It follows from \cref{3.AlmostAnd*} that the families A-$\Delta_\ell^*$, A-$\Delta_{\ell,0}^*$, and A-IP$_0^*$ introduced in the previous sections coincide, respectively, with the families $(\Delta_\ell\text{-rich})^*$,  $(\Delta_{\ell,0}\text{-rich})^*$, and (IP$_0$-rich)$^*$.
For the rest of the paper, when dealing with the families $(\Delta_\ell\text{-rich})^*$,  $(\Delta_{\ell,0}\text{-rich})^*$, (IP-rich)$^*$, and (IP$_0$-rich)$^*$, we will find it convenient to  denote these families, correspondingly, by A-$\Delta_\ell^*$, A-$\Delta_{\ell,0}^*$, A-IP$^*$, \index{A-IP$^*$}  and A-IP$_0^*$ (where "A" stands for "almost").
\end{rem}
\subsection{Relations between various notions  of largeness}
The following diagram (where $\ell_1,\ell_2\in\N$ and $\ell_1<\ell_2$) presents in a unified way the relations between the various families of sets which were introduced in the previous sections. 
$$\begin{matrix}
\text{IP}_0& \not\supseteq  &   \Delta_{1}\text{-rich}
&   \subsetneq   &   \Delta_1   &   \subsetneq & \Delta_{1,0}    & \supsetneq &\text{IP}_0&&\\
\requal&  &   \rsupsetneq
&      &   \rsupsetneq   &   & \rsupsetneq    &  & \rsupsetneq&& \\
\text{IP}_0&     \not\supseteq  &   \Delta_{2}\text{-rich}
&   \subsetneq   &   \Delta_2   &   \subsetneq & \Delta_{2,0}    &  \not\supseteq & \text{IP}&\subsetneq&\text{IP}_0\\
\requal&  &   \rsupsetneq
&      &   \rsupsetneq   &   & \rsupsetneq    &  & \requal&&\requal \\
\vdots&  &   \vdots
&      &   \vdots   &   & \vdots   &  & \vdots&&\vdots \\
\requal&  &   \rsupsetneq
&      &   \rsupsetneq   &   & \rsupsetneq    &  & \requal&&\requal \\
\text{IP}_0&     \not\supseteq  &   \Delta_{\ell_1}\text{-rich}
&   \subsetneq   &   \Delta_{\ell_1}   &   \subsetneq & \Delta_{\ell_1,0}    &  \not\supseteq & \text{IP}&\subsetneq &\text{IP}_0\\
\requal&  &   \rsupsetneq
&      &   \rsupsetneq   &   & \rsupsetneq    &  & \requal&&\requal \\
\vdots&  &   \vdots&
       & \vdots &   & \vdots &  & \vdots&&\vdots \\
\requal&  &   \rsupsetneq
&      &   \rsupsetneq   &   & \rsupsetneq    &  & \requal&&\requal \\
\text{IP}_0&     \not\supseteq  &   \Delta_{\ell_2}\text{-rich}
&   \subsetneq   &   \Delta_{\ell_2}   &   \subsetneq & \Delta_{\ell_2,0}    &  \not\supseteq & \text{IP}&\subsetneq &\text{IP}_0\\
\requal&  &   \rsupsetneq
&      &   \rsupsetneq   &   & \rsupsetneq    &  & \requal&&\requal \\
\vdots&  &   \vdots&
       & \vdots &   & \vdots &  & \vdots&&\vdots 
\end{matrix}
$$
In what follows we will provide explanations/proofs for the non-obvious inclusions presented in the diagram above.\\ 

We begin with noting that by \cref{2.OddPolynomialCharacterization}, for any $\ell_1,\ell_2\in\N$ with $\ell_1<\ell_2$, we can find an odd polynomial $v(x)\in\R[x]$ with $\deg(v_1)=2\ell_2-1$ and an $\epsilon>0$ such that $\mathcal R(v,\epsilon)$ is $\Delta_{\ell_2,0}^*$ but not $\Delta_{\ell_1}^*$. It follows that $\N\setminus \mathcal R(v,\epsilon)$ is a $\Delta_{\ell_1}$ (and hence, a $\Delta_{\ell_1,0}$) set but not a $\Delta_{\ell_2,0}$ set. Hence, 
by the definitions of the families $\Delta_\ell$-rich, $\Delta_\ell$ and $\Delta_{\ell,0}$, the following diagram holds for any $\ell_1<\ell_2$:
$$\begin{matrix}
\Delta_{\ell_1}\text{-rich}&\subseteq&\Delta_{\ell_1}&\subseteq&\Delta_{\ell_1,0}\\
\rsupseteq& &\rsupsetneq& &\rsupsetneq \\
\Delta_{\ell_2}\text{-rich}&\subseteq&\Delta_{\ell_2}&\subseteq&\Delta_{\ell_2,0}
\end{matrix}
$$
We will show now that the following relations hold for any $\ell\in\N$: 
$$\begin{matrix}
\text{IP}_0&\not\supseteq&\Delta_{\ell}\text{-rich}\\
& &\rsupsetneq\\
& &\Delta_{\ell+1}\text{-rich}
\end{matrix}
$$
\begin{lem}\label{3.Delta^lRichNoIPs}
For each $\ell\in\N$,  there exists a $\Delta_\ell\text-\text{rich}$ set which is not an $\text{IP}_0$ set nor a $\Delta_{\ell'}$ set for any $\ell'>\ell$.
\end{lem}
\begin{proof}
First we will show that, given $\ell\in\N$ and $D\subseteq \N$ with  $d^*(D)=\delta>0$, for any $M\geq 2^\ell$, $n_1,...,n_M\in\N$ and any $E\subseteq\N$ with $d^*(E)=0$, there exists an $n\in D$ for which 
\begin{equation}\label{3.DcapEisEmpty}
\{\partial(n_{j_1},...,n_{j_{2^\ell-1}},n)\,|\,1\leq j_{1}<j_2<j_3<\cdots< j_{2^\ell-1}\leq M\}\cap E=\emptyset.
\end{equation}
To prove the contrapositive, note that for any $m_1,m_2,m_3,\cdots,m_{2^\ell}\in\Z$,
\begin{multline*}
\partial(m_1,...,m_{2^\ell})=\partial(m_{2^{\ell}-2^{\ell-1}+1},...,m_{2^\ell})-\partial(m_1,...,m_{2^{\ell}-2^{\ell-1}})\\
=\partial(m_{2^{\ell}-2^{\ell-2}+1},...,m_{2^\ell})-\partial(m_{2^{\ell}-2^{\ell-1}+1},...,m_{2^\ell-2^{\ell-2}})-\partial(m_1,...,m_{2^{\ell}-2^{\ell-1}})\\
=\partial(m_{2^{\ell}-2^{\ell-2}+1},...,m_{2^\ell})-\sum_{t=0}^{1}\partial(m_{2^{\ell}-2^{\ell-t}+1},...,m_{2^{\ell}-2^{\ell-t-1}})\\
=\cdots=m_{2^\ell}-\sum_{t=0}^{\ell-1}\partial(m_{2^{\ell}-2^{\ell-t}+1},...,m_{2^{\ell}-2^{\ell-t-1}}).
\end{multline*}
So if \eqref{3.DcapEisEmpty} does not hold for any $n\in D$, we have 
$$D\subseteq\bigcup_{m\in E}\{m+\sum_{t=0}^{\ell-1}\partial(n_{j_{2^{\ell}-2^{\ell-t}+1}},...,n_{j_{2^{\ell}-2^{\ell-t-1}}})\,|\,1\leq j_1<\cdots<j_{2^\ell-1}\leq M\}$$
and hence $d^*(E)\geq \frac{\delta}{M^{2^\ell}}>0$.\\

Now let $\ell\in\N$, let $E\subseteq \N$ with $d^*(E)=0$ and let $\alpha_0,...,\alpha_\ell\in\R\cap [0,1)$ be irrational numbers such that $1,\alpha_0,...,\alpha_\ell$ are rationally independent. By Weyl's equidistribution theorem  \cite[Theorem 16]{weyl1916Mod1}, the sequence 
$$(n\alpha_0,n^2\alpha_0,n\alpha_1,n^2\alpha_1,...,n\alpha_\ell,n^2\alpha_\ell),\text{ }n=1,2,...$$
is uniformly distributed on $\mathbb T^{2(\ell+1)}$. Hence, by \eqref{3.DcapEisEmpty}, we can choose inductively an increasing sequence $(n_k)_{k\in\N}$ in $\N$ such that
\begin{enumerate}[(1)]
\item $D_\ell((n_k)_{k\in\N})\cap E=\emptyset$.
\item For each $j\in\{1,...,\ell\}$,
$$\lim_{k\rightarrow\infty}n_k\alpha_j=0.$$
\item For each $j\in\{1,...,\ell\}$, $$-2^{j-1}\binom{2j+1}{2}\lim_{k\rightarrow\infty}n^2_k\alpha_j=\alpha_{j-1}.$$
\item $\lim_{k\rightarrow\infty}n_k\alpha_0=\frac{1}{2}.$
\end{enumerate}
Let  $p\in\beta\N$ be a non-principal ultrafilter with  $\{n_k\,|\,k\in\N\}\in p$.
Since $p$ satisfies the hypothesis of \cref{2.TechnicalLemma}, we have
$$\plimgG{p_\ell}{n}{\N}n^{2\ell+1}\alpha_\ell=\plimgG{p}{n}{\N}n\alpha_0=\frac{1}{2}.$$
It follows that $\{n\in\N\,|\,\|n^{2\ell+1}\alpha_\ell-\frac{1}{2}\|<\frac{1}{4}\}\in p_\ell$. By \cref{1.DeltaCharacterization}, (ii),  there exists an increasing sequence $(m_k)_{k\in\N}$ such that $\{m_k\,|\,k\in\N\}\subseteq \{n_k\,|\,k\in\N\}$ and 
$$D_\ell((m_{k})_{k\in\N})\subseteq \{n\in\N\,|\,\|n^{2\ell+1}\alpha_\ell-\frac{1}{2}\|<\frac{1}{4}\}.$$ Thus, for each $n\in D_\ell((m_{k})_{k\in\N})$,
$$\|n^{2\ell+1}\alpha_\ell-\frac{1}{2}\|<\frac{1}{4}$$
and hence 
$$\|n^{2\ell+1}\alpha_\ell\|\geq\frac{1}{4}.$$
(Note that $(m_k)_{k\in\N}$ is a subsequence of $(n_k)_{k\in\N}$ and so $D_\ell((m_k)_{k\in\N})\subseteq  D_\ell((n_k)_{k\in\N})$.)\\

It follows that for any $E\subseteq \N$ with $d^*(E)=0$ we can find an increasing sequence $(m_k)_{k\in\N}$ in $\N$ for which
$$D_\ell((m_k)_{k\in\N})\subseteq \{n\in\N\,|\,\|n^{2\ell+1}\alpha_\ell\|\geq\frac{1}{4}\}\setminus E$$
This implies that the set 
$$\{n\in\N\,|\,\|n^{2\ell+1}\alpha_\ell\|\geq\frac{1}{4}\}$$
is $\Delta_\ell\text-$rich. However,  by \cref{2.OddDegreeRecurrence} and \cref{3.IP0Returns}, it is not an $\text{IP}_0$ set nor a $\Delta_{\ell'}$ set for any $\ell'>\ell$.
\end{proof}
\begin{rem}
Let  $v(x)$ be an even polynomial with no constant term and at least one irrational coefficient and let $\epsilon\in (0,\frac{1}{3})$. The  argument used in the  proof of \cref{2.LackOfRecurrenceForEvenPowers}  shows that for any $\ell\in\N$, there exists an increasing sequence $(n_k)_{k\in\N}$ in $\N$ such that $D_\ell((n_k)_{k\in\N})\cap \mathcal R(v,\epsilon)=\emptyset$. Furthermore,  the proof of \cref{3.Delta^lRichNoIPs} shows that for any given $E\subseteq\N$ with $d^*(E)=0$, one can choose $(n_k)_{k\in\N}$ so that $D_\ell((n_k)_{k\in\N})\cap E=\emptyset$. Thus,  the set 
$$\{n\in\N\,|\,\|v(n)\|>\epsilon\}$$ is $\Delta_\ell\text-$rich for each $\ell\in\N$. On the other hand, by \cref{3.IP0Returns}, this set is not an $\text{IP}_0$ set. 
\end{rem}
\begin{question}
In \cite[Section 2]{AlmostIPBerLeib}, it was shown that \text{\rm{A-IP$_0^*\not\supseteq$A-IP$^*$}}. Hence $\text{IP-rich}\not\supseteq\text{IP}_0\text{-rich}.$
Is it true that for any $\ell\in\N$, 
$\Delta_{\ell}\text{-rich}\not\supseteq \Delta_{\ell,0}\text{-rich}$?
\end{question}
\begin{cor}\label{3.FirstImportantApplicationA}
Let $\ell\in\N$. The following statements hold:
\begin{enumerate}[(i)]
\item For any $r\in\N$ large enough, there exists a set $E\subseteq \N$ such that $E$ is \text{\rm{IP$_r^*$}} but not $\Delta_{\ell,R}^*$ for any $R\geq 2^\ell$.
\item  There exists a set $E\subseteq \N$ such that $E$ is \text{\rm{A-IP$_0^*$}} but not  \text{\rm{A-$\Delta_{\ell,0}^*$}}. 
\end{enumerate}
\end{cor}
\begin{proof}
We will only show (i), the proof of (ii) is similar. By \cref{3.Delta^lRichNoIPs}, there exists a set $D\subseteq \N$ which is  a $\Delta_\ell$-rich set  but not an IP$_0$ set. In particular, $D$ is  $\Delta_{\ell,0}$ but not IP$_0$. Thus, there exists $r_0\in\N$ such that for any $r\geq r_0$, $D$ is not an IP$_{r_0}$ set. It follows that for any $r\geq r_0$, the set $E=\N\setminus D$ is an \text{\rm{IP$_{r}^*$}} set but not a $\Delta_{\ell,R}^*$ set for any $R\geq 2^\ell$.
\end{proof}
\begin{rem}
Let $v\in\R[x]$ be an odd polynomial of degree $2\ell-1$. The goal of this remark is to stress that, when $\epsilon$ is small enough, the  mere fact that $\mathcal R(v,\epsilon)$ is IP$_r^*$ for some $r\in\N$, does not necessarily imply that  $\mathcal R(v,\epsilon)$ is $\Delta_{\ell,R}^*$ for each $R\geq 2^\ell$.
Indeed, for any $r_0\in\N$ there exists a small enough $\epsilon>0$ and an $r>r_0$ for which the set $\mathcal R(v,\epsilon)$ is IP$_r^*$ but not IP$_{r_0}^*$. Let $r$ and $E$ be as in \cref{3.FirstImportantApplicationA},(i). By picking $\epsilon$ small enough, we see that for some $r'>r$, both  $\mathcal R(v,\epsilon)$ and $E$ are IP$_{r'}^*$,  but $E$ is not a  $\Delta_{\ell,R}^*$ for each $R\geq 2^\ell$.
\end{rem}
Now we prove that for each $\ell\in\N$,
$$\Delta_\ell\text{-rich}\subsetneq \Delta_\ell.$$
\begin{lem}\label{3.DensityZeroDeltaL}
Let $\ell\in\N$. Any $\Delta_\ell$ set contains a $\Delta_\ell$ set with zero upper Banach density.
\end{lem}
\begin{proof}
Let $D\subseteq\N$ be a $\Delta_\ell$ set and let $(n_k)_{k\in\N}$ be an increasing sequence such that $D_\ell((n_k)_{k\in\N})\subseteq D$. By passing to a subsequence, if needed, we assume without loss of generality that $10n_k<n_{k+1}$ for each $k\in\N$. We claim that for any $k\in\N$ and any $N,M\in\N$ with $0<N-M<n_k$, the set $I=\{M+1,...,N\}$ satisfies 
\begin{equation}\label{3.DensityZeroEquation}
|I\cap D_\ell((n_k)_{k\in\N})|\leq 3^k.
\end{equation}
Indeed, fix $k,N,M\in\N$ with $0<N-M<n_k$, and suppose, for the sake of contradiction, that \eqref{3.DensityZeroEquation} does not hold. Observe that there exists an $R>k$ such that for each $n\in I\cap D_\ell((n_k)_{k\in\N})$ one can find  $\xi^{(n)}_1,...,\xi^{(n)}_R\in\{-1,0,1\}$ satisfying
$$n=\sum_{t=1}^R\xi^{(n)}_tn_t.$$
Note that the set $\{\sum_{t=1}^k\xi_tn_t\,|\,\xi_1,...,\xi_k\in\{-1,0,1\}\}$ has at most $3^k$ elements. Thus, by the  pigeonhole principle, there exist distinct $a,b\in I\cap D_\ell((n_k)_{k\in\N})$ such that
$$a=\sum_{t=k+1}^R\xi^{(a)}_tn_t+\sum_{t=1}^k\xi^{(a)}_tn_t\text{ and }b=\sum_{t=k+1}^R\xi^{(b)}_tn_t+\sum_{t=1}^k\xi^{(a)}_tn_t.$$
So, since $b\neq a$, $|b-a|=|\sum_{t=k+1}^R(\xi^{(b)}_t-\xi^{(a)}_t)n_t|>n_k$. This contradicts the fact that for any $n,m\in I$, $|n-m|<n_k$.\\

It now follows from \eqref{3.DensityZeroEquation}, that for each $k\in\N$, if $n_k\leq N-M<n_{k+1}$, then 
$$\frac{|D_\ell((n_k)_{k\in\N})\cap\{M+1,...,N\}|}{N-M}\leq \frac{3^{k+1}}{n_{k}}$$
and hence
$$d^*(D_\ell((n_k)_{k\in\N})=\limsup_{R-L\rightarrow\infty}\frac{|D_\ell((n_k)_{k\in\N})\cap\{L+1,...,R\}|}{R-L}\leq \frac{3^{k+1}}{n_{k}}.$$
Thus, $d^*(D_\ell((n_k)_{k\in\N})=0$, completing the proof.
\end{proof}
Next we show that for any $\ell\in\N$,  $$\Delta_\ell\subsetneq \Delta_{\ell,0}.$$ 
\begin{lem}\label{3.StrictDeltaEllinclusion}
There exists a set $E\subseteq\N$ that is a  $\Delta_{\ell,0}$ set for each $\ell\in\N$ but  is not a $\Delta_\ell$ set for any $\ell\in\N$.
\end{lem}
\begin{proof}
For each $k\in\N$, let $E_k=\{(2k)!,(2k)!2...,(2k)!k\}$
and let 
$$E=\bigcup_{k\in\N}E_k.$$
Let $\ell\in\N$. Since for any $r\geq 2^\ell$, there exists an $R\geq r$ for which  $E_{R}$ is a $\Delta_{\ell,r}$ set, $E$  is a $\Delta_{\ell,0}$ set. It only remains to show that $E$ does not contain any $\Delta_1$ set (this will imply that $E$ contains no $\Delta_\ell$ set for any $\ell\in\N$).\\
Given a  $\Delta_1$ set $D$ there exists an increasing sequence $(n_k)_{k\in\N}$ in $\N$ such that  $D_1((n_k)_{k\in\N})\subseteq D$. Note that for such a sequence
\begin{equation}\label{3.LargeGap0NotNormalProof}
(n_{k}-n_{1})-(n_{k}-n_2)=n_2-n_1
\end{equation}
for each $k\geq 3$. Since  $\max E_s<\min E_{s+1}-(2s)!$ for any $s\in\N$, we have that for $n,m\in E$ large enough, $|n-m|>n_2-n_1$. It follows from \eqref{3.LargeGap0NotNormalProof} that $D$ can not be a subset of $E$, which completes the proof.
\end{proof}
\begin{rem}
The set $E$ in the proof of \cref{3.StrictDeltaEllinclusion}, is an IP$_0$ set which is not an IP set. Hence
$$\text{IP}\subsetneq \text{IP}_0.$$
\end{rem}
The next two results show that
$$ \text{IP}_0\subsetneq \Delta_{1,0}.$$
\begin{lem}\label{3.EveryDeltaIsIP}
Any IP$_0$ set is a $\Delta_{1,0}$ set.
\end{lem}
\begin{proof}
The proof is similar to the well known fact that any IP set is a $\Delta$ set (see for example \cite[Lemma 9.1]{FBook}). We will actually show that any IP$_r$ set contains a $\Delta_{1,r}$ set.\\
Let $r\geq2$  and let $\Gamma$ be an IP$_r$ set containing FS$((n_k)_{k=1}^r)$ for some $n_1,...,n_r\in\N$. For each $k\in\{1,...,r\}$, let  $s_k=n_1+n_2+\cdots+n_k$. Then for any  $k>l$, $s_k-s_l\in \Gamma$. Thus, $\Gamma$ is a $\Delta_{1,r}$ set.  
\end{proof}
The result contained in the following short lemma is similar to a remark made in the Introduction to \cite{BergelsonErdosDifferences} . 
\begin{lem}\label{3.DeltaWithNoIP}
The set
$$D=D_1((10^k)_{k\in\N})=\{9\sum_{s=i}^j10^s\,|\,1\leq i\leq j\},$$
is a $\Delta_1$ set but not an IP$_3$ set.
\end{lem}
\begin{proof}
Let $x,y,z\in D$. Suppose that $x\leq y\leq z$. If $x+y\in D$ and $y+z\in D$, we have, by analysing the decimal expansions of $x$, $y$, $z$, $x+y$, and $y+z$, that $x+z\not\in D$. So, $D$ is not an IP$_3$ set.
\end{proof}
Finally, we prove that for $\ell\geq 2$,
$$\Delta_{\ell,0}\not\supseteq \text{IP}.$$
We will denote by $\mathcal F$ the set of all finite non-empty subsets of $\N$.
\begin{thm}\label{3.IPWithNoDelta2}
Let $(n_k)_{k\in\N}$ be an increasing sequence of natural numbers. Suppose that for any $\alpha,\beta,\gamma\in\mathcal F$, 
\begin{equation}\label{3.injectiveCondition}
\sum_{j\in\alpha}n_j+\sum_{j\in\beta}n_j=\sum_{j\in\gamma}n_j\text{ if and only if }\alpha\cup\beta=\gamma\text{ and }\alpha\cap\beta=\emptyset.
\end{equation}
Then for any $\ell\geq 2$, $\text{FS}((n_k)_{k\in\N})$ is not a $\Delta_{\ell,2^{\ell-2}14}$ set.\footnote{
Condition \eqref{3.injectiveCondition} holds for any sequence $(n_k)_{k\in\N}$ in $\N$ which satisfies $\frac{n_{k+1}}{n_k}\geq 3$ for every $k\in\N$.
}
\end{thm}
\begin{proof}
By the definition of $\Delta_{\ell,r}$ sets, all we need to show is that given a 14-element sequence $(c_k)_{k=1}^{14}$ in $\Z$  with $D_2((c_k)_{k=1}^{14})\subseteq\N$, 
$D_2((c_k)_{k=1}^{14})\not\subseteq \text{FS}((n_k)_{k\in\N})$. Assume for contradiction that 
$D_2((c_k)_{k=1}^{14})\subseteq \text{FS}((n_k)_{k\in\N})$. For any 
$k_1,k_2\in\{3,...,14\}$ with $k_1 < k_2$, let $\alpha (k_1,k_2)\in\mathcal F$ be such that
\begin{equation}\label{3.SubstractFirstTerms}
(c_{k_2}-c_{k_1})-(c_2-c_1)=\sum_{j\in\alpha(k_1,k_2)}n_j.
\end{equation}
Since $D_2((c_k)_{k=1}^{14})\subseteq \text{FS}((n_k)_{k\in\N})$, for any $k_1,...,k_4$ which satisfy  $3\leq k_1<k_2<k_3<k_4\leq 14$, there exists $\alpha(k_1,k_2,k_3,k_4)\in\mathcal F$ such that
\begin{equation}\label{3.FallingInside}
(c_{k_4}-c_{k_3})-(c_{k_2}-c_{k_1})=\sum_{j\in\alpha(k_1,k_2,k_3,k_4)}n_j.
\end{equation}
It follows from \eqref{3.SubstractFirstTerms} and \eqref{3.FallingInside} that 
\begin{equation}\label{3.GeneralDifference}
\sum_{j\in\alpha(k_3,k_4)}n_j=\sum_{j\in\alpha(k_1,k_2,k_3,k_4)}n_j+\sum_{j\in\alpha(k_1,k_2)}n_j.
\end{equation}
Thus, we get from \eqref{3.GeneralDifference} and \eqref{3.injectiveCondition} that $$\alpha(k_1,k_2)\subseteq \alpha(k_3,k_4).$$
Consider now the set $A=\{c_5-c_4,c_8-c_7,c_{11}-c_{10},c_{14}-c_{13}\}$ and note that $|A|= 4$, (otherwise we would have that $0\in\N$). Hence, at least two of the elements of $A$ are either strictly positive or  strictly negative. Without loss of generality, we will assume that $c_5-c_4,c_8-c_7>0$.\\
Let $\lambda_1=\alpha(3,5)\setminus\alpha(3,4)$ and let $\rho_1=\alpha(3,4)\setminus\alpha(3,5)$, then 
$$c_5-c_4=(c_5-c_3)-(c_4-c_3)=\sum_{j=\alpha(3,5)}n_j-\sum_{j=\alpha(3,4)}n_j=\sum_{j\in\lambda_1}n_j-\sum_{j\in\rho_1}n_j.$$
(Note that a priori $\lambda_1$ or $\rho_1$ could be empty. We follow the convention that $\sum_{j\in\emptyset}n_j=0$.)\\
Since $c_5-c_4>0$, we must have that $\lambda_1\neq\emptyset$. A similar argument shows that if we let $\lambda_2=\alpha(6,8)\setminus\alpha(6,7)$ and $\rho_2=\alpha(6,7)\setminus\alpha(6,8)$, then 
$$c_8-c_7=\sum_{j\in\lambda_2}n_j-\sum_{j\in\rho_2}n_j$$
and hence $\lambda_2\neq \emptyset$.\\
Since $\alpha(3,5)\cup\alpha(3,4)\subseteq\alpha(6,8)\cap\alpha(6,7)$, the sets $\lambda_1,\lambda_2,\rho_1,\rho_2$ are pairwise disjoint. Let $\beta\in\mathcal F$ be such that 
$$\sum_{j\in\beta}n_j=(c_8-c_7)-(c_5-c_4)\in\text{FS}((n_k)_{k\in\N}),$$
then 
\begin{equation}\label{3.DisjointSum}
\sum_{j\in\beta}n_j+\sum_{j\in\lambda_1\cup\rho_2}n_j=\sum_{j\in\lambda_2\cup\rho_1}n_j.
\end{equation}
By noting that $\lambda_1\cup\rho_2\not\subseteq \lambda_2\cup\rho_1$, we see that \eqref{3.DisjointSum} contradicts \eqref{3.injectiveCondition}. This completes the proof. 
\end{proof}
\begin{cor}\label{3.SecondImportantApplication}
Let $\ell\geq 2$. For any $r\in\N$ large enough there exists a set $E\subseteq \N$ such that $E$ is $\Delta_{\ell,R}^*$ but not \text{\rm{IP$_{R}^*$}} for any $R\in\N$.
\end{cor}
\begin{proof}
The proof is similar to the proof of \cref{3.FirstImportantApplicationA}.
\end{proof}
\begin{question}
Is it true that $\text{A-}\Delta_{\ell,0}^*\not\subseteq\text{A-IP}^*$?
\end{question}
\printindex
\bibliography{Bib}
\bibliographystyle{plain}
\end{document}